\documentclass[a4paper,british]{amsart}
\usepackage[T1]{fontenc}
\usepackage[latin9]{inputenc}
\usepackage{babel}
\usepackage{refstyle}
\usepackage{float}
\usepackage{mathtools}
\usepackage{amstext}
\usepackage{amsthm}
\usepackage{amssymb}
\usepackage[unicode=true,
 bookmarks=true,bookmarksnumbered=false,bookmarksopen=false,
 breaklinks=false,pdfborder={0 0 1},backref=false,colorlinks=false]
 {hyperref}
\hypersetup{pdftitle={Positive jacobian constraints for elliptic boundary value problems with piecewise-regular coefficients arising from hybrid inverse problems},
 pdfauthor={Yves Capdeboscq Tianrui Dai}}

\makeatletter

\AtBeginDocument{\providecommand\secref[1]{\ref{sec:#1}}}
\AtBeginDocument{\providecommand\figref[1]{\ref{fig:#1}}}
\AtBeginDocument{\providecommand\eqref[1]{\ref{eq:#1}}}
\AtBeginDocument{\providecommand\lemref[1]{\ref{lem:#1}}}
\AtBeginDocument{\providecommand\thmref[1]{\ref{thm:#1}}}
\AtBeginDocument{\providecommand\subsecref[1]{\ref{subsec:#1}}}
\AtBeginDocument{\providecommand\exref[1]{\ref{ex:#1}}}
\AtBeginDocument{\providecommand\assuref[1]{\ref{assu:#1}}}
\AtBeginDocument{\providecommand\proref[1]{\ref{pro:#1}}}
\AtBeginDocument{\providecommand\claimref[1]{\ref{claim:#1}}}
\AtBeginDocument{\providecommand\propref[1]{\ref{prop:#1}}}
\pdfpageheight\paperheight
\pdfpagewidth\paperwidth

\newcommand{\noun}[1]{\textsc{#1}}
\RS@ifundefined{subsecref}
  {\newref{subsec}{name = \RSsectxt}}
  {}
\RS@ifundefined{thmref}
  {\def\RSthmtxt{theorem~}\newref{thm}{name = \RSthmtxt}}
  {}
\RS@ifundefined{lemref}
  {\def\RSlemtxt{lemma~}\newref{lem}{name = \RSlemtxt}}
  {}

\numberwithin{equation}{section}
\numberwithin{figure}{section}
\theoremstyle{definition}
\newtheorem*{example*}{\protect\examplename}
\theoremstyle{plain}
\newtheorem{thm}{\protect\theoremname}
\theoremstyle{plain}
\newtheorem{assumption}[thm]{\protect\assumptionname}
\theoremstyle{plain}
\newtheorem{lem}[thm]{\protect\lemmaname}
\theoremstyle{definition}
\newtheorem{defn}[thm]{\protect\definitionname}
\theoremstyle{remark}
\newtheorem{rem}[thm]{\protect\remarkname}
\theoremstyle{remark}
\newtheorem*{rem*}{\protect\remarkname}
\theoremstyle{plain}
\newtheorem{prop}[thm]{\protect\propositionname}
\theoremstyle{plain}
\newtheorem*{prop*}{\protect\propositionname}
\theoremstyle{definition}
\newtheorem{example}[thm]{\protect\examplename}
\theoremstyle{remark}
\newtheorem{claim}[thm]{\protect\claimname}
\theoremstyle{plain}
\newtheorem{criterion}[thm]{\protect\criterionname}
\theoremstyle{plain}
\newtheorem*{fact*}{\protect\factname}
\theoremstyle{plain}
\newtheorem*{thm*}{\protect\theoremname}

\newcommand{\oublie}[1]{}

\usepackage{tikz}
\usetikzlibrary{math}
\usepackage{a4wide}

\makeatother

\providecommand{\assumptionname}{Assumption}
\providecommand{\claimname}{Claim}
\providecommand{\criterionname}{Criterion}
\providecommand{\definitionname}{Definition}
\providecommand{\examplename}{Example}
\providecommand{\factname}{Fact}
\providecommand{\lemmaname}{Lemma}
\providecommand{\propositionname}{Proposition}
\providecommand{\remarkname}{Remark}
\providecommand{\theoremname}{Theorem}

\begin{document}
\global\long\def\R{\mathbb{R}}%
\global\long\def\ii{\mathrm{i}}%
\global\long\def\N{\mathbb{N}}%
\global\long\def\supp{{\rm supp}}%
\global\long\def\span{\operatorname{span}}%
\global\long\def\div{\operatorname{div}}%
\global\long\def\phi{\varphi}%
\global\long\def\epsilon{\varepsilon}%

\global\long\def\Laplacian{\Delta}%

\global\long\def\dstar{d^{\star}}%

\global\long\def\one{\mathbf{1}}%
\global\long\def\consi{\mathbf{i}}%

\global\long\def\log{\text{ln}}%

\global\long\def\rank{\operatorname{rank}}%

\global\long\def\Id{\text{Id}}%

\global\long\def\C{\mathcal{C}}%

\global\long\def\Holder{H\ddot{o}lder}%

\global\long\def\Poincare{Poincar\acute{e}}%

\global\long\def\ker{\text{ker}}%

\global\long\def\Im{\text{Im}}%

\global\long\def\HH{\mathbf{H}}%
\global\long\def\JacJ{\mathbf{J}}%
\global\long\def\JacSJ{\mathbf{J_{f}}}%

\global\long\def\translation{\kappa}%

\global\long\def\translationconstant{\vartheta}%

\global\long\def\pr{\text{pr}}%

\title[Jacobian constraints with discontinuous coefficients]{Positive jacobian constraints for elliptic boundary value problems
with piecewise-regular coefficients arising from multi-wave inverse
problems}
\author{Yves Capdeboscq}
\author{Tianrui Dai}
\address{Université Paris Cité, CNRS, Sorbonne Université, Laboratoire Jacques-Louis
Lions (LJLL), F-75006 Paris, France}
\email{yves.capdeboscq@u-paris.fr}
\email{tianrui.dai@etu.u-paris.fr}
\begin{abstract}
Multi-wave inverse problems are indirect imaging methods using the
interaction of two different imaging modalities. One brings spatial
accuracy, and the other contrast sensitivity. The inversion method
typically involve two steps. The first step is devoted to accessing
internal datum of quantities related to the unknown parameters being
observed. The second step involves recovering the parameters themselves
from the internal data. To perform that inversion, a typical requirement
is that the Jacobian of fields involved does not vanish. A number
of authors have considered this problem in the past two decades, and
a variety of methods have been developed. Existing techniques require
Hölder continuity of the parameters to be reconstructed. In practical
applications, the medium may present embedded elements, with distinct
physical properties, leading to discontinuous coefficients. In this
article we explain how a Jacobian constraint can imposed in the piecewise
regular case, when the physical model is a divergence form second
order linear elliptic boundary value problem.
\end{abstract}

\thanks{This study contributes to the IdEx Université de Paris ANR-18-IDEX-0001.}
\maketitle

\section{Introduction}

Parameter reconstruction problems for elliptic boundary value problems
are indirect reconstruction problems with, at best, logarithmic stability
\cite{zbMATH03998383,MANDACHE-01}. While these type of measurements
are desirable as they are non intrusive, and typically require low
cost apparels, such weak stability implies that only low-resolution
reconstruction can be achieved in practice \cite{HANDBOOK-MMI2-2015,WANG-ANASTASIO-HANDBOOK-2011,WIDLAK-SCHERZER-2012}.
The Calderón problem for electrical impedance tomography (EIT) \cite{CALDERON-1980,UHLMANNIP-2009,zbMATH06457660},
the inverse scattering problems \cite{COLTON-KRESS-98} and optical
tomography \cite{1999-arridge} are the main examples of such problems.
The stability of such methods dramatically improve when, instead of
making absolute measurements \textit{ex nihilo}, they are used to
estimate perturbations of a know medium \cite{AMMARI-KANG-04,CAPDEBOSCQ-VOGELIUS-03A}.
On the other hand, fast wave imaging modalities, such as ultrasound
tomography or MRI, preserve singularities and achieve excellent spatial
accuracy, at the cost of a loss of quantitative information with respect
to the amplitude of the parameters involved.

Over the past two decades, coupled-physics, or multi-wave, or hybrid
inverse problems (the final commonly accepted name is yet to be determined)
have emerged. These imaging modalities aim to benefit from the advantages
of both approaches: one for accurate contrast estimations, and the
other for high resolution \cite{alberti-capdeboscq-2018,2017-ammari-et-al,BAL-2012,kuchment-2012}.
The most developed hybrid modality is photo-acoustic tomography (PAT)
\cite{WOS:000571255900003,KUCHMENT-KUNYANSKY-HANBOOK-2011,2009-razansky,WANG-ANASTASIO-HANDBOOK-2011},
in which light and ultrasounds are combined. Many other modalities
have been considered, all combining a diffusive process with a much
less diffusive one \cite{zbMATH06924651,AMMARI-BONNETIER-CAPDEBOSCQ-TANTER-FINK-2008,AMMARI-CAPDEBOSCQ-DEGOURNAY-ROZANOVA-TRIKI-2011,hasanov-2004,KUCHMENT-KUNYANSKY-HANBOOK-2011,LAVANDIER-JOSSINET-GATHIGNOL-00,seo-kim-etal-2012,seo-woo-2011,WANG-ANASTASIO-HANDBOOK-2011,ZHANG-WANG-2004}.

The parameter reconstruction method in all these problems start with
a data collection step, where some internal data is reconstructed,
involving both the parameter of interest and the solution of the PDE
involving this parameter. In PAT, the internal data is $\mu(x)u(x)$
, where $\mu$ is the optical absorption and $u$ is the light intensity.
In Current Density Impedance Imaging (CDII), the internal data is
$\left|\gamma(x)\nabla u(x)\right|$, where $\gamma$ is the conductivity
and $u$ is the electric field. The second step involves extracting
the parameter from this data, ($\mu(x)$ in PAT, $\gamma$ in CDII).
The mathematical problem considered in this article is related to
this second step.

\begin{example*}[A Jacobian constraint example]
\label{ex:scalar reconstructing} Consider the problem of reconstructing
$\gamma$, a scalar function, in 
\[
-\div(\gamma Du)=0\quad\text{in }\Omega,
\]
from the knowledge of the potential $u$ in $\Omega$, as in \cite{alessandrini-1986}.
It appears in a variety of contexts, such has Hydrology \cite{neuman-yakowitz-1979},
CDII \cite{bal-guo-monard-2014,NTT-Rev-11,SCOTT-JOY-ARMSTRONG-HENKELMAN-1991,WOO-LEE-SY-MUN-1994}
and Acousto--Electric Tomography \cite{zbMATH06924651,AMMARI-BONNETIER-CAPDEBOSCQ-TANTER-FINK-2008,CAPDEBOSCQ-FEHRENBACH-DEGOURNAY-KAVIAN-09}.
If $\gamma$ is regular, we have
\begin{equation}
D(\log\gamma)\cdot Du=-\Delta u\quad\text{in }\Omega.\label{eq:log a}
\end{equation}
Suppose given $d$ measurements $u_{1},\dots,u_{d}$. By (\ref{eq:log a})
we obtain 
\[
\left(D(\log\gamma)\right)^{T}\begin{bmatrix}Du_{1}, & \cdots & ,Du_{d}\end{bmatrix}=-\begin{bmatrix}\Delta u_{1}, & \cdots & ,\Delta u_{d}\end{bmatrix}\quad\text{in }\Omega.
\]
If $\det\begin{bmatrix}Du_{1}, & \cdots & ,Du_{d}\end{bmatrix}>0$
holds true, then $\nabla(\log\gamma)$, and in turn $\gamma$ up to
a multiplicative constant, are explicitly readable from the data by
inverting the matrix $\begin{bmatrix}Du_{1}, & \cdots & ,Du_{d}\end{bmatrix}$.
More generally, given $N\geq d$ measurements, the least-square optimisation
problem associated to the (possibly overdetermined) system of equations
\[
\left(D(\log\gamma)\right)^{T}\begin{bmatrix}Du_{1}, & \cdots & ,Du_{N}\end{bmatrix}=-\begin{bmatrix}\Delta u_{1}, & \cdots & ,\Delta u_{N}\end{bmatrix}\quad\text{in }\Omega.
\]
has a unique minimiser when $\det\begin{bmatrix}Du_{i_{i}}, & \cdots & ,Du_{i_{d}}\end{bmatrix}>0$
for some $\left(i_{1},\ldots,i_{d}\right)\in\left\{ 1,\ldots,N\right\} ^{d}$.
\end{example*}
Using unique continuation methods, it is possible to address the parameter
reconstruction problem without imposing Jacobian constraints \cite{alessandrini-2014,2015-alessandrini-dicristo-francini-vessella,zbMATH07537315,zbMATH07468759,zbMATH07096506}.
On the other hand, when non-vanishing constraints are satisfied, the
stability estimates are optimal (of Lipschitz type) and often lead
to explicit reconstruction formulae \cite{alberti-capdeboscq-2018,BAL-2012}.

The focus of this paper is non-vanishing Jacobian constraints. In
two dimensions, for the conductivity equation, a generalisation of
the Radó--Kneser--Choquet theorem \cite{alessandrini-1986,ALESSANDRINI-NESI-01,ALESSANDRINI-NESI-2015,Bauman-Marini-Nesi-2001}
shows that imposing a non-vanishing Jacobian constraint globally,
and independently of the conductivity is possible : in practice, it
suffices to verify that the Jacobian doesn't vanish when the conductivity
is equal to one everywhere. Such an approach cannot be extended to
three dimensions, or more general elliptic problems \cite{CAPDEBOSCQ-15,WOOD-1991}.
Suitable solutions can be constructed using complex geometrical optics
solutions (CGOS) \cite{BAL-2012,BAL-BONNETIER-MONARD-TRIKI-2013,BAL-REN-2011,BAL-UHLMANN-2010,BAL-UHLMANN-2013},
but this construction depends on the unknown coefficients, which must
be smooth and isotropic. Another approach \cite{alberti-capdeboscq-2018,BAL-UHLMANN-2013}
is based on the Runge approximation \cite{LAX-1956,MALGRANGE-1955-56}.
It is valid for all PDE for which a Unique Continuation Property holds,
it allows for anisotropic coefficients, and the smoothness assumptions
are precisely that for the Unique Continuation Property of the underlying
equation, namely Lipschitz regularity or Hölder continuity depending
on the equation. By combining this approach with the Whitney projection
method, it is proved in \cite{2020-Alberti-Capdeboscq-IMRN} that
the set of suitable solutions is open and dense, with explicit estimates
on the number of solutions needed. A very related result, using a
slightly different Whitney projection argument, was proved independently
around the same time \cite{CEKIC-LIN-RULAND-2020}. Very recently,
another approach was proposed, which showed that choosing random boundary
values was possible \cite{Alberti_2022}.

All these methods rely on some regularity of the coefficients. In
practical cases, it is desirable to consider the case of piecewise
regular coefficients, each region corresponding to a different strata
in geology, or a different organ in medical imaging.

In this work, we show how the approach introduced in \cite{2020-Alberti-Capdeboscq-IMRN}
can be extended to the case of piecewise regular coefficients. We
use existing unique continuation results within the regular parts
of the domain, \cite{zbMATH03168004,LAX-1956,MALGRANGE-1955-56},
and introduce adequate quantities to cross over discontinuities. These
constructions may prove useful for other models where the principal
part is in divergence form. 

In \secref{Main-results} we detail our assumptions, state the main
result of this article, and explain its proof, using intermediate
results proved in the subsequent sections. In \cite{2020-Alberti-Capdeboscq-IMRN},
Hölder continuity is crucially used in two instances : to show existence
of solutions satisfying the adequate constraints via Runge Approximation,
and to use the Whitney projection method, which is based on Sard's
lemma, which itself uses Hölder continuity. As a result, our developments
come in two parts. In \secref{construction} we establish the existence
a finite number of solutions such that the non vanishing Jacobian
constraint is satisfied in the whole domain : this requires adapting
existing unique continuation results to cross smooth interfaces. In
\secref{Whitney} we use the continuity of fluxes across interfaces
resulting from the divergence structure of the principal part, via
appropriate charting, to deduce non-vanishing properties of gradients
up to the internal subregion boundaries. 

\section{Model, Assumptions and Main Results\label{sec:Main-results}}

\subsection{Problem definition\label{subsec:model-and-domain}}

The ambient space is $\mathbb{R}^{d}$, with $d\geq2$.
\begin{assumption}
\label{assu:Piecewise_assumption_Domain} Assume that $\Omega$ is
an open, bounded and connected domain in $\mathbb{R}^{d}$ with a
$C^{2}$ boundary.

Assume that $\Omega$ contains $N\geq1$ open connected disjoint sets
$\Omega_{1},\ldots,\Omega_{N}$ with $C^{2}$ boundaries such that
$0<\text{d\ensuremath{\left(\cup_{\ell=1}^{N}\Omega_{i},\mathbb{R}^{d}\setminus\Omega\right)}}$.

Assume furthermore that for any $i\in\left\{ 1,\dots,N\right\} ,$
$\Omega_{i}$ has a $C^{2}\left(\mathbb{R}^{d-1}\right)$ boundary,
and each connected component of its boundary is in common with at
most one other $\Omega_{j},$ $j\neq i$.

We write $\Omega_{N+1}=\Omega\setminus\overline{\left(\cup_{i=1}^{N}\Omega_{i}\right)}$,
and denote $\Gamma_{ij}=\partial\Omega_{i}\cap\partial\Omega_{j}$
when this set is non-empty.

Additionally, assume that each $\Gamma_{ij}$ is sphere-like, that
is, there exists an open neighbourhood $U_{ij}$ of $\Gamma_{ij}$,
an open neighbourhood $V_{ij}$ of $S^{d-1}$, and a $C^{2}$ diffeomorphism
$\psi_{ij}:U_{ij}\rightarrow V_{ij}$ such that $\psi_{ij}\left(\Gamma_{ij}\right)=S^{d-1}$.

Moreover, there exists $d_{0}>0$ such that
\begin{equation}
\forall i,j\in\left\{ 1,\ldots,N+1\right\} ^{2},i\neq j,\text{ if }\partial\Omega_{i}\setminus\Gamma_{ij}\neq\emptyset\text{ then }d\left(\text{\ensuremath{\Gamma_{ij},\partial\Omega_{i}\setminus\Gamma_{ij}}}\right)>d_{0}.\label{eq:defkappa}
\end{equation}
\end{assumption}

An example of such a configuration is given in \figref{5inclusionstart}.

Following the usual notation, given a set $U$ we write
\[
\one_{U}:x\to\begin{cases}
1 & \text{ if }x\in U,\\
0 & \text{otherwise.}
\end{cases}
\]

\begin{assumption}
\label{assu:regularitecoeff_} Given $\alpha\in\left(0,1\right]$,
for each $i\in\left\{ 1,\cdots,N+1\right\} $ let $A_{i}\in C^{0,\alpha}\left(\mathbb{R}^{d};\mathcal{M}_{d}^{s}\left(\mathbb{R}\right)\right)$
be a symmetric-matrix-valued function which is uniformly elliptic,
that is, there exists $\lambda>0$ such that for all $x\in\Omega$,
and all $\zeta\in\mathbb{R}^{d},$
\begin{equation}
\lambda\left|\zeta\right|^{2}<A\left(x\right)\zeta\cdot\zeta.\label{eq:elliptic}
\end{equation}
For each $i\in\left\{ 1,\cdots,N+1\right\} $, let $b_{i}\in C^{0,\alpha}\left(\mathbb{R}^{d};\mathbb{R}^{d}\right)$,
$c_{i}\in C^{0,\alpha}\left(\mathbb{R}^{d};\mathbb{R}^{d}\right)$
and $q_{i}\in C^{0,\alpha}\left(\mathbb{R}^{d};\mathbb{R}\right)$,
be such that 
\[
\max\left(\left\Vert A_{i}\right\Vert _{_{C^{0,\alpha}\left(\mathbb{R}^{d};\mathbb{R}^{d\times d}\right)},}\left\Vert b_{i}\right\Vert _{_{C^{0,\alpha}\left(\mathbb{R}^{d};\mathbb{R}^{d}\right)}},\left\Vert c_{i}\right\Vert _{_{C^{0,\alpha}\left(\mathbb{R}^{d};\mathbb{R}^{d}\right)}},\left\Vert q_{i}\right\Vert _{_{C^{0,\alpha}\left(\mathbb{R}^{d};\mathbb{R}\right)}}\right)\leq\lambda^{-1},
\]
where for $n\geq1$,
\[
\left\Vert f\right\Vert _{_{C^{0,\alpha}\left(\mathbb{R}^{d};\mathbb{R}^{n}\right)}}=\sup_{\mathbb{R}^{d}}\left|f\right|+\sup_{{x\neq y\in\mathbb{R}^{d}\atop 0\neq\zeta\in\mathbb{R}^{n}}}\frac{\left|f\left(x\right)\cdot\zeta-f\left(y\right)\cdot\zeta\right|}{\left|x-y\right|\left|\zeta\right|}.
\]
Finally, when $d\geq3$, we assume additionally that $A_{i}\in C^{0,1}\left(\mathbb{R}^{d};\mathcal{M}_{d}^{s}\left(\mathbb{R}\right)\right)$\footnote{So that the Unique Continuation Property holds in each subdomain.
This assumption can be relaxed when $\left(A_{i}^{k\ell}(x)\right)_{1\leq k,\ell\leq d}=\left(a_{i}(x)\delta_{k\ell}\right)_{1\leq k,\ell\leq d}$
for all $x$.}.

We write, for all $x\in\Omega\setminus\cup_{i,j}\Gamma_{ij}$, 
\begin{equation}
A=\sum_{i=1}^{N+1}A_{i}\one_{\Omega_{i}},\quad b=\sum_{i=1}^{N+1}b_{i}\one_{\Omega_{i}},\quad c=\sum_{i=1}^{N+1}c_{i}\one_{\Omega_{i}},\text{ and }q=\sum_{i=1}^{N+1}q_{i}\one_{\Omega_{i}}.\label{eq:defAbcq}
\end{equation}
\end{assumption}

We consider a second order elliptic operator of the form $L:u\to-\div\left(ADu+bu\right)+c\cdot Du+qu$,
and the PDE under consideration is 
\begin{equation}
Lu=0\quad\text{in }\Omega.\label{eq:PDE}
\end{equation}

Thanks to assumption~\ref{assu:regularitecoeff_} the weak solutions
of \eqref{PDE} enjoy additional regularity within each subdomain
$\Omega_{i}$, $i=1,\cdots,N+1$. \lemref{regula-global} follows
from classical regularity results, see e.g. \cite[Theorem 5.19 and 5.20]{GIAQUINTA-93}
for a modern exposition.
\begin{lem}
\label{lem:regula-global}If $u\in H^{1}\left(\Omega\right)$ is a
weak solution of \eqref{PDE}, such that there exists $g\in C^{1,\alpha}\left(\overline{\Omega}\right)$
such that $u-g\in H_{0}^{1}\left(\Omega\right)$ and for any $v\in H_{0}^{1}\left(\Omega\right)$
there holds 
\[
\int_{\Omega}ADu\cdot Dv\text{ d}x+\int_{\Omega}ub\cdot Dv\text{ d}x+\int_{\Omega}vc\cdot Du\text{ d}x+\int_{\Omega}quv\text{ d}x=0.
\]
 Then, 
\begin{equation}
u\in H^{1}\left(\Omega\right)\cap\left(\cup_{i=1}^{N+1}C^{1,\alpha}\left(\Omega_{i}\right)\right)=:H\left(\Omega\right),\label{eq:defhom}
\end{equation}
 and 
\[
\sum_{i=1}^{N+1}\left(\left\Vert Du\right\Vert _{C^{0,\alpha}\left(\Omega_{i}\right)}+\left\Vert u\right\Vert _{C^{0,\alpha}\left(\Omega_{i}\right)}\right)\leq C\left(\left\Vert u\right\Vert _{L^{2}\left(\Omega\right)}+\left\Vert g\right\Vert _{C^{1,\alpha}\left(\overline{\Omega}\right)}\right),
\]
where the constant $C$ depends on $\lambda$ given in assumption~\ref{assu:regularitecoeff_}
and $\Omega_{i}$, $i=\left\{ 1,\cdots,N+1\right\} $ only.
\end{lem}

We are now in position to define the quantity of interest in this
paper.
\begin{defn}[\emph{Non-vanishing Jacobian solutions}]
\label{def:NVJacobianSolutions} Given $P\geq d+1$, we call $\left\{ u_{i}^{x_{0}}\right\} _{i=1}^{P}\in H\left(\Omega\right)^{P}$
(a group of) non-vanishing Jacobian solutions of \eqref{PDE} at $x_{0}\in\Omega\setminus\cup_{i,j}\Gamma_{ij}$,
if
\begin{enumerate}
\item for $i=1,\ldots,P$ there holds $Lu_{i}^{x_{0}}=0$ in $\Omega$,
\item The solutions $\left\{ u_{i}^{x_{0}}\right\} _{i=1}^{P}\in H\left(\Omega\right)^{P}$
satisfy $\rank\left(\JacJ\left(u_{1}^{x_{0}},\cdots,u_{P}^{x_{0}}\right)\left(x_{0}\right)\right)=d+1$,
where 
\[
\JacJ\left(u_{1},\ldots,u_{P}\right)\left(x\right)\coloneqq\left(\begin{array}{c}
Du_{1}\\
\vdots\\
Du_{P}
\end{array}\begin{array}{c}
u_{1}\\
\vdots\\
u_{P}
\end{array}\right)\left(x\right)=\left(\begin{array}{ccc}
\partial_{1}u_{1} & \ldots & \partial_{d}u_{1}\\
\vdots & \vdots & \vdots\\
\partial_{1}u_{P} & \ldots & \partial_{d}u_{P}
\end{array}\begin{array}{c}
u_{1}\\
\vdots\\
u_{P}
\end{array}\right)\left(x\right).
\]
\end{enumerate}
\end{defn}

\begin{rem}
Thanks to \lemref{regula-global}, pointwise values of $\JacJ\left(u_{1}^{x_{0}},\cdots,u_{P}^{x_{0}}\right)$
are well defined at any $x\in\Omega\setminus\cup_{i,j}\Gamma_{ij}$.
The use of the word `Jacobian' for the quantity $\JacJ$ may seem
abusive. Indeed one would expect a Jacobian to be 
\[
\left(\begin{array}{c}
Dv_{1}\\
\vdots\\
Dv_{d}
\end{array}\right)=\left(\begin{array}{ccc}
\partial_{1}v_{1} & \ldots & \partial_{d}v_{1}\\
\vdots & \vdots & \vdots\\
\partial_{1}v_{d} & \ldots & \partial_{d}v_{d}
\end{array}\right),
\]
 for some function $v_{1},\cdots,v_{d}$. It turns out that the slightly
generalised Jacobian we consider is a natural quantity to consider
in this problem, to take into account the behaviour of solution across
interfaces. On the other hand, from a family of non-vanishing Jacobian
solutions, one can extract a subfamily $\left(u_{i_{1}}^{x_{0}},\cdots,u_{i_{d}}^{x_{0}}\right)$
such that $\det\left(Du_{i_{1}}^{x_{0}},\cdots,Du_{i_{d}}^{x_{0}}\right)\left(x_{0}\right)\neq0,$
so it encompasses non-vanishing Jacobian constraints for the traditional
definition of a Jacobian.
\end{rem}

Following the strategy introduced in \cite{2020-Alberti-Capdeboscq-IMRN}
we define the admissible set for an integer $P$
\begin{multline*}
\mathcal{A}(P):=\left\{ \left(u_{1},u_{2},\cdots,u_{P}\right)\in H(\Omega)^{P}:\forall x\in\Omega\setminus\cup_{i,j}\Gamma_{ij},\right.\\
\left.\left(u_{1},u_{2},\cdots,u_{P}\right)\text{ are non vanishing Jacobian solutions}\right\} .
\end{multline*}

For a geometrical reason that will be discussed later, we introduce
the notation 
\[
\dstar=\begin{cases}
d & \text{when }d=2,4,8\\
d+1 & \text{otherwise when }d\geq3.
\end{cases}
\]

\subsection{Main result}

The main result of this article is the following.
\begin{thm}
\label{thm:mainresult}Under assumption~\ref{assu:Piecewise_assumption_Domain}
and assumption~\ref{assu:regularitecoeff_}, when $P\geq\left[\frac{d+\dstar+1}{\alpha}\right]$,
$\mathcal{A}\left(P\right)$ is an open and dense subset of $H\left(\Omega\right)^{P},$
where $H(\Omega)$ is defined in \eqref{defhom}.
\end{thm}

\begin{rem*}
This theorem is an extension to the piecewise regular context \cite[Theorem 2.3]{2020-Alberti-Capdeboscq-IMRN}.
In terms of the result itself, the number $P$ obtained in \cite{2020-Alberti-Capdeboscq-IMRN}
is $\left[\frac{2d}{\alpha}\right]$, thus our result requires a slightly
larger number than the regular case; however the number of subdomain
where the coefficients are regular does not play a role.

A careful reader comparing \cite[Theorem 2.3]{2020-Alberti-Capdeboscq-IMRN}
and \thmref{mainresult} might notice that our result applies to the
whole domain, instead of a compact subset. Another simplification
is that we need not assume that the Dirichlet (or Neuman or Robin)
boundary value problem associated to (\ref{eq:PDE}) is well posed
for our result to hold.
\end{rem*}
The proof is done in several steps.
\begin{thm}
\label{thm:piecewise_regular-construction} For any $\sigma>0$ there
exists $\epsilon>0$ such that for any $x\in\Omega\setminus\cup_{i,j}\Gamma_{ij}$,
there exists $d+1$ solutions denoted as $u_{1}^{x},u_{2}^{x},\cdots,u_{d+1}^{x}$
such that $u_{i}^{x}\in H^{1}\left(\Omega\right)$ and $Lu_{i}^{x}=0$
in $\overline{\Omega}$ for $i\in\left\{ 1,2,\cdots,d+1\right\} $,
and there holds 
\begin{equation}
\det\JacJ\left(u_{1}^{x},u_{2}^{x},\cdots,u_{d+1}^{x}\right)\left(y\right)=\det\begin{bmatrix}\begin{array}{rrrr}
\partial_{1}u_{1}^{x} & \cdots & \partial_{d}u_{1}^{x} & u_{1}^{x}\\
\vdots & \vdots & \vdots & \vdots\\
\partial_{1}u_{d+1}^{x} & \cdots & \partial_{d}u_{d+1}^{x} & u_{d+1}^{x}
\end{array}\end{bmatrix}\left(y\right)>\sigma\label{eq:requirement}
\end{equation}

for any $y\in B\left(x,\epsilon\right)\cap\Omega_{j}$, $j\in\left\{ 1,\ldots,N+1\right\} $.
\end{thm}

This result is proved in \secref{construction}. It does not follow
directly from classical unique continuation arguments, because of
the discontinuous nature of the coefficients of \eqref{PDE}.

Choose $\sigma=1$, and let $\epsilon$ be the corresponding ball
radius. We may extract a finite cover of $\overline{\Omega}$ from
$\cup_{x\in\Omega\setminus\cup_{i,j}\Gamma_{ij}}B\left(x,\epsilon\right)$,
of cardinality smaller than, say, $\left(\text{\ensuremath{\epsilon^{-1}}diam}\left(\Omega\right)\right)^{d}+1.$
As a result,
\begin{equation}
\mathcal{A}\left(\left[\left(\frac{\text{diam}\left(\Omega\right)}{\epsilon}\right)^{d}\right]+1\right)\neq\emptyset.\label{eq:RoughNumber}
\end{equation}
To reduce the cardinality of the required group of non-vanishing Jacobian
solutions, and to prove the density property we announced, we use
a Whitney reduction lemma.

This strategy was used in \cite{2020-Alberti-Capdeboscq-IMRN}, based
on a method introduced in \cite{Greene-Wu-1975a}, and used the Hölder
continuity of the Jacobian map $\JacJ$. In our setting, $\JacJ$
may be discontinuous across interfaces $\Gamma_{ij}$.

On the other hand, because of the divergence form of the principal
part of the elliptic operator $L$, a mixed-type (for lack of a better
word) Jacobian map of the form 
\[
\left(A\nabla u\cdot h_{1}+b\cdot h_{1}u,\nabla u\cdot h_{2},\cdots,\nabla u\cdot h_{d},u\right),
\]
with appropriately chosen $\left(h_{1},\cdots,h_{d}\right)\in C^{0,1}\left(\Omega;\mathbb{R}^{d\times d}\right)$
is continuous.
\begin{prop}
\label{pro:MNT-Light} There exists a family of vector-valued functions
$\mathcal{F}=f_{1},\cdots,f_{d^{\star}}\in C^{0,1}\left(\Omega;\mathbb{R}^{d}\right)^{d^{\star}}$,

such that
\begin{enumerate}
\item For every $x\in\Omega$, there holds $\rank\left(f_{1},\cdots,f_{\dstar}\right)\left(x\right)=d.$
\item On each $\Gamma_{ij},$$\left|f_{1}\right|=1$, $f_{1}$ is normal
to $\Gamma_{ij}$, and $f_{1}\cdot f_{j}=0$ for any $j\geq2.$
\item For any $u\in H\left(\Omega\right)$ weak solution of \eqref{PDE},
the map 
\begin{equation}
\JacSJ\left(u,\mathcal{F}\right):=\left(\left(ADu+bu\right)\cdot f_{1},Du\cdot f_{2},\cdots,Du\cdot f_{\dstar},u\right)\label{eq:defJacFlow}
\end{equation}
satisfies $\JacSJ\left(u,\mathcal{F}\right)\in C^{0,\alpha}\left(\Omega;\mathbf{\text{\ensuremath{\R}}}^{\dstar+1}\right).$
\end{enumerate}
\end{prop}

This proposition is proved in \secref{Whitney}.
\begin{rem*}
The vector $f_{1}$ can be thought of as the extension of the normal
vector and $f_{2},\cdots,f_{\dstar}$ as the tangent vectors on each
boundary $\Gamma_{ij}$. Indeed, since $A$ and $b$ are only piecewise
regular, only the normal flux is continuous (and, in turn, Hölder
continuous) across interfaces between any $\Omega_{i}$ and $\Omega_{j}$.
This forces $f_{2},\cdots,f_{\dstar}$ to be tangent to the interface.
A topological difficulty arises in all dimensions, except $2,4$ and
$8$, which requires the introduction of an extra element to obtain
a full rank family of Lipschitz continuous tangent vectors. This classical
result \cite{bott1958parallelizability,kervaire1958non} is discussed
further in \secref{Kervaire}.
\end{rem*}
To untangle the dependence of $\JacSJ$ on $u$ and $\mathcal{F}$,
we reformulate $\JacSJ$ as follows.
\begin{prop}
\label{prop:current-matrix}We note $P_{d,d+1}\in\mathbb{R}^{d\times\left(d+1\right)}$
the projection from $\mathbb{R}^{d+1}$ to $\mathbb{R}^{d}$ given
by $P_{d,d+1}$ is such that $\left(P_{d,d+1}\right)_{ij}=\delta_{ij}$.
We note $E_{d+1,d}$ the extension from $\mathbb{R}^{d}$ to $\mathbb{R}^{d+1}$
given by $E_{d+1,d}$ is such that $\left(E_{d+1,d}\right)_{ij}=\delta_{ij}$.

Set
\begin{eqnarray*}
T:\left(\Omega\setminus\cup_{ij}\Gamma_{ij}\right)\times\mathbb{R}^{d+1} & \to & \mathcal{L}\left(\mathbb{R}^{\left(d+1\right)\times\left(d^{\star}+1\right)}\right)\\
\left(x,\zeta_{1},\cdots,\zeta_{\dstar+1}\right) & \to & \left(\begin{array}{ccccc}
A^{T}(x)P_{d,d+1}\zeta_{1} & P_{d,d+1}\zeta_{2} & \cdots & P_{d,d+1}\zeta_{\dstar} & P_{d,d+1}\zeta_{\dstar+1}\\
b(x)P_{d,d+1}\zeta_{1} & 0 & \cdots & 0 & 1
\end{array}\right)
\end{eqnarray*}

For any $x\in\Omega\setminus\cup_{ij}\Gamma_{ij}$, and for any $\left(\xi_{1},\cdots,\xi_{\dstar}\right)\in\left(\mathbb{R}^{d}\right)^{\dstar}$
there holds 
\begin{equation}
\text{rank}\left(T\left(x,E_{d+1,d}\xi_{1},\cdots,E_{d+1,d}\xi_{\dstar},\text{e}_{d+1}\right)\right)=\text{rank}\left(\xi_{1},\cdots,\xi_{\dstar}\right)+1.\label{eq:rqnkT}
\end{equation}

Furthermore, we have 
\[
\JacSJ\left(u,\mathcal{F}\right)=\left(\partial_{1}u,\cdots,\partial_{d}u,u\right)T\left(x,E_{d+1,d}f_{1},\cdots,E_{d+1,d}f_{\dstar},\text{e}_{d+1}\right),
\]
 where $\JacSJ$ is given by \eqref{defJacFlow}.
\end{prop}

\begin{proof}
The last column of $T\left(x,E_{d+1,d}\xi_{1},\cdots,E_{d+1,d}\xi_{\dstar},\text{e}_{d+1}\right)$
is $e_{d+1}\neq0$. Together with the fact that $P_{d,d+1}E_{d+1,d}=I_{d}$,
the identity matrix in $\mathbb{R}^{d}$, the first $\dstar$ columns
are
\[
\begin{pmatrix}A^{T}(x)\xi_{1} & \xi_{2} & \cdots & \xi_{\dstar}\\
b(x)\xi_{1} & 0 & \cdots & 0
\end{pmatrix}
\]
Thanks to the uniform ellipticity of $A$, $A^{T}\xi_{1}\cdot\xi_{1}>\lambda\left|\xi_{1}\right|^{2}$
and \eqref{rqnkT} follows. The identity involving $\JacSJ$ is straightforward.
\end{proof}
The Whitney reduction argument is as follows.
\begin{lem}
\label{lem:Whitney-Reduction-Lemma-Light} Given $P\in\mathbb{N}$
large enough so that $\mathcal{A}\left(P\right)\neq\emptyset$, define
\begin{eqnarray}
F:\Omega\setminus\cup_{i,j}\Gamma_{ij}\times\mathbf{\mathbb{R}}^{\dstar+1} & \rightarrow & \mathbb{R}^{P}\nonumber \\
\left(x,\zeta\right) & \to & F_{x}\zeta\label{eq:defFWhitney}
\end{eqnarray}
where 
\[
F_{x}\zeta:=\begin{bmatrix}\left(\partial_{1}u_{1},\cdots,\partial_{d}u_{1},u_{1}\right)\\
\vdots\\
\left(\partial_{1}u_{P},\cdots,\partial_{d}u_{P},u_{P}\right)
\end{bmatrix}T\left(x,E_{d+1,d}f_{1},\cdots,E_{d+1,d}f_{\dstar},\text{e}_{d+1}\right)\zeta,
\]
with $\left\{ u_{1},\cdots,u_{P}\right\} \in\mathcal{A}\left(P\right).$
Then $F_{x}$ has rank $d+1$. For $P>\frac{d+\dstar+1}{\alpha}$,
and $a\in\mathbf{\R}^{P-1}$, let $P_{a}$ be the map from $\mathbb{R}^{P}$
to $\mathbb{R}^{P-1}$ defined by 
\[
P_{a}(y)=(y_{1}-a_{1}y_{P},\cdots,y_{P-1}-a_{P-1}y_{P})
\]
for $y=(y_{1},y_{2},\cdots,y_{P})\in\mathbb{R}^{P}$ . Let $G=\left\{ a\in\R^{P-1}|P_{a}\circ F_{x}\text{ has rank \ensuremath{d+1}}\right\} ,$
then $\vert\R^{P-1}-G\vert_{\text{Lebesgue}}=0.$
\end{lem}

The proof of this lemma is given in \subsecref{PfPaFx}. We then translate
this reduction result for $\JacSJ$ into its counterpart for our original
target map $\JacJ$.
\begin{lem}
\label{lem:Main-reduction-lemma} Given any $P>\frac{d+\dstar+1}{\alpha}$,
and any $\left\{ u_{1},\cdots,u_{P}\right\} \in\mathcal{A}\left(P\right)$,
let $G$ be the set of $a=\left(a_{1},\cdots,a_{P-1}\right)\in\mathbb{R}^{P-1}$
such that for all $\ensuremath{x\in}\Omega\setminus\cup_{ij}\Gamma_{ij}$
there holds
\[
\rank\JacJ\left(u_{1}-a_{1}u_{P},\cdots,u_{P-1}-a_{P-1}u_{P}\right)\left(x\right)=d+1.
\]

Then $\left|\R^{P-1}\setminus G\right|_{lebesgue}=0$.
\end{lem}

\begin{proof}
Given $\mathcal{F}=\left\{ f_{1}(x),\cdots,f_{\dstar}(x)\right\} \in\left(C^{0,1}\left(\Omega;\mathbb{R}^{d}\right)\right)^{\dstar}$
as defined in proposition~\ref{pro:MNT-Light}, for $x\in\Omega\setminus\cup_{ij}\Gamma_{ij}$,
let $F_{x}:\R^{\dstar+1}\rightarrow\mathbb{R}^{P}$ as given in \eqref{defFWhitney}.
Thanks to \lemref{Whitney-Reduction-Lemma-Light}, we have rank $F_{x}=d+1$
and for a.e $a\in\R^{P-1}$, $P_{a}\circ F_{x}$ has rank $d+1$ which
means 
\[
\text{rank \ensuremath{\left[\begin{array}{c}
\JacSJ\left(u_{1}-a_{1}u_{P},\mathcal{F}\right)\\
\vdots\\
\JacSJ\left(u_{P-1}-a_{P-1}u_{P},\mathcal{F}\right)
\end{array}\right]\left(x\right)}}=d+1.
\]
Denote $\mathcal{\mathcal{J}}=$$\JacJ\left(u_{1}-a_{1}u_{P},\cdots,u_{P-1}-a_{P-1}u_{P}\right)$
and $\mathcal{T}=T\left(x,E_{d+1,d}f_{1},\cdots,E_{d+1,d}f_{d^{\star}},\text{e}_{d+1}\right)$
so that 
\[
\left[\begin{array}{c}
\JacSJ\left(u_{1}-a_{1}u_{P},\mathcal{F}\right)\\
\vdots\\
\JacSJ\left(u_{P-1}-a_{P-1}u_{P},\mathcal{F}\right)
\end{array}\right]=\mathcal{JT}.
\]
Then, $\text{rank\ensuremath{\left(\mathcal{JT}\right)}}=d+1$, and
since $\rank\left(\mathcal{J}\mathcal{T}\right)\leq\min\left(\rank\left(\mathcal{J}\right),\text{rank\ensuremath{\left(\mathcal{T}\right)}}\right)$,
we conclude that $d+1\geq\rank\left(\mathcal{J}\right)\geq d+1$,
which proves that $\text{rank\ensuremath{\left(\mathcal{J}\right)=d+1}}.$
\end{proof}
With the above lemma, we have now returned to a familiar setting,
where no further complications due to the discontinuous nature of
the coefficients arise. The rest of the proof of the \thmref{mainresult}
now follows an argument similar to the one found in \cite[Theorem 2.3]{2020-Alberti-Capdeboscq-IMRN},
and a variant of the argument above to prove that the set $\mathcal{A}\left(P\right)$
is open, which we include in \secref{EndProofMR}. 

\subsection{Application on an example}

We revisit example~\exref{scalar reconstructing}, namely the reconstruction
of the conductivity from the knowledge of the solution to illustrate
how our result naturally extends existing results derived for uniformly
regular parameters. In addition to assumption~\assuref{Piecewise_assumption_Domain}
and assumption~\assuref{regularitecoeff_}, suppose that $b=c=q=0,$
$A=\gamma I_{d}$ , where $\gamma$ is scalar valued function, and
$\alpha=1.$ 
\begin{prop*}
Given $P>0$ such that, $\mathcal{A}\left(P\right)\neq\emptyset$,
and $\left\{ u_{1},\cdots,u_{P}\right\} \in\mathcal{A}\left(P\right).$
For each $\ell\in\left\{ 1,\cdots,P\right\} $, $u_{\ell}\in BV\left(\Omega\right)$,
and its singular part is a jump set. The union over $\ell$ of these
jump sets is $\cup_{i,j}\Gamma_{ij}.$ 

Given $x\in\Gamma_{ij},$ let $n\left(x\right)$ be the normal pointing
from $\Omega_{i}$ to $\Omega_{j},$ that is, $x+tn\left(x\right)\in\Omega_{i}$
for $t<0$ and $x+tn\left(x\right)\in\Omega_{j}$ for $t>0$, provided
$t$ is small enough. 

Let $u_{p}$ be such that $\lim_{t\to0^{+}}\left|Du_{p}\left(x+tn\right)\right|=\max_{k\in\left\{ 1,\cdots,P\right\} }\lim_{t\to0^{+}}\left|Du_{k}\left(x+tn\right)\right|$.
Then 
\[
\lim_{t\to0^{+}}\ln\left|Du_{p}\left(x+tn\left(x\right)\right)\cdot n\left(x\right)\right|-\ln\left|Du_{p}\left(x-tn\left(x\right)\right)\cdot n\left(x\right)\right|=-\left[\ln\gamma\left(x\right)\right]_{ij},
\]
where 
\[
\left[\ln\gamma\left(x\right)\right]_{ij}=\lim_{{h\to x\atop h\in\Omega_{j}}}\ln\gamma\left(h\right)-\lim_{{h\to x\atop h\in\Omega_{i}}}\ln\gamma\left(h\right).
\]

The absolutely continuous part of $D\ln\gamma$ with respect to the
Lebesgue measure is determined by 
\[
\begin{vmatrix}Du_{1}\\
\vdots\\
Du_{P}
\end{vmatrix}D\ln\gamma=\begin{vmatrix}\Delta u_{1}\\
\vdots\\
\Delta u_{P}
\end{vmatrix}\text{ on }\Omega_{k},\quad k=\left\{ 1,\cdots,N+1\right\} .
\]
\end{prop*}
\begin{rem*}
In particular, $\gamma$ is uniquely determined up to a multiplicative
constant.
\end{rem*}
\begin{proof}
Thanks to \lemref{regula-global}, and proposition~\proref{MNT-Light},
there holds
\[
\JacSJ\left(u_{\ell},\mathcal{F}\right)=\left(\gamma Du_{\ell}\cdot f_{1},Du_{\ell}\cdot f_{2},\cdots,Du_{\ell}\cdot f_{d^{*}},u_{\ell}\right)\in C^{0,1}\left(\Omega\right).
\]
Because $\rank\mathcal{F}=d,$ the discontinuities of $Du_{\ell}$
are included in the $\Gamma_{ij}$. For any given $\ell$, it may
not correspond exactly, to the entire $\cup_{i,j}\Gamma_{ij}$ since
$Du_{\ell}\cdot f_{1}$ may possibly vanish on these interfaces; however,
\[
\text{rank \ensuremath{\left[\JacSJ\left(u_{1},\mathcal{F}\right),\cdots,\JacSJ\left(u_{P},\mathcal{F}\right)\right]\left(x\right)}}=d+1\text{ for all }x\in\Omega,
\]
thus in particular, $\rank\left[\left(\gamma Du_{1}\cdot f_{1},\cdots,\gamma Du_{P}\cdot f_{1}\right)\right]\left(x\right)=1$
on $\cup_{i,j}\Gamma_{ij},$ thus the set is indeed the whole $\cup_{i,j}\Gamma_{ij}$
. Equipped with all interfaces $\Gamma_{ij}$, and a set of associated
normal vectors, we may recover the jumps between the different regions. 

Thanks to proposition~\proref{MNT-Light}, $\gamma Du_{\ell}\cdot f_{1}\in C^{0,1}\left(\Omega\right),$and
$f_{1}\left(x\right)=n(x)\text{ on each \ensuremath{\Gamma_{ij}} }$.

Since $\lim_{t\to0^{+}}\left|Du_{p}\left(x+tn\right)\right|=\max_{k\in\left\{ 1,\cdots,P\right\} }\lim_{t\to0^{+}}\left|Du_{k}\left(x+tn\right)\right|,$and
not all such limit can be zero by since $(u_{1},\cdots,u_{P})\in\mathcal{A}\left(P\right),$
$\lim_{t\to0}\ln\left|\left(\gamma Du_{p}\right)\left(x+tn\right)\cdot n\right|\in\mathbb{R}.$
In particular, 
\[
\lim_{t\to0^{+}}\ln\left|\left(\gamma Du_{p}\right)\left(x+tn\right)\cdot n\right|-\ln\left|\left(\gamma Du_{p}\right)\left(x-tn\right)\cdot n\right|=0,
\]
 and therefore 
\[
\lim_{t\to0^{+}}\ln\left|Du_{p}\left(x+tn\right)\cdot n\right|-\ln\left|Du_{p}\left(x-tn\right)\cdot n\right|=-\left[\ln\gamma\left(x\right)\right]_{ij}.
\]

The final identity is obtained exactly as in the regular case.
\end{proof}

\section{\label{sec:construction}Proof of \thmref{piecewise_regular-construction}}

We construct a group of solutions which satisfies the Jacobian constraint
locally within one subdomain $\Omega_{i}$ and extend them one subdomain
at a time. To do this rigorously, we introduce a construction map,
and an associated index map\emph{,} defining the order in which the
extension is performed\emph{.} For any permutation $\consi:\left\{ 1,\cdots,N+1\right\} \to\left\{ 1,\cdots,N+1\right\} $,
we denote $\Omega_{I_{k}}=\Omega_{\mathbf{\consi}(1)}\cup\cdots\cup\Omega_{\consi(k)}$
for $k\in\left\{ 1,\cdots,N+1\right\} $. We have the following definition:
\begin{defn}
\label{def:construction-map} We say a permutation $\consi:\left\{ 1,\cdots,N+1\right\} \to\left\{ 1,\cdots,N+1\right\} $
is a \emph{construction map, }if the following holds :
\begin{quotation}
For any $j\in\left\{ 2,\cdots,N+1\right\} $, there exists a unique
$\mathbf{k}\left(j\right)\in\left\{ 1,\cdots,j-1\right\} $ such that
$\partial\Omega_{\consi\left(j\right)}\cap\partial\Omega_{I_{j-1}}=\Gamma_{\consi\left(j\right)\consi\left(\mathbf{k}\left(j\right)\right)}$
.
\end{quotation}
With $\consi$ a construction map comes $\text{\ensuremath{\mathbf{j_{i}}=\left\{ \text{\ensuremath{\mathbf{j_{i}}^{1}}},\cdots,\ensuremath{\mathbf{j_{i}}^{N+1}}\right\} }}$
the \emph{map index} of $\consi$ defined as follows:
\begin{enumerate}
\item for every $s\in\left\{ 1,\cdots,N+1\right\} $, $\mathbf{j_{i}}^{s}\in\left\{ 1,\cdots,N+1\right\} ^{N+1}$,
\item The starting map $\mathbf{j_{i}}^{1}$ satisfies $\mathbf{j_{i}}^{1}=\left(\consi\left(1\right),\cdots,\consi\left(1\right)\right)$
\item For any $s\in$$\left\{ 2,\cdots,N+1\right\} $ we have $\left(\mathbf{j_{i}}^{s}\right)_{\consi\left(\ell\right)}=\left(\mathbf{j_{i}}^{s-1}\right)_{\mathbf{i}\left(\ell\right)}$
if $\ensuremath{\ell\leq s-1}$, $\left(\mathbf{j_{i}}^{s}\right)_{\consi\left(\ell\right)}=\consi\left(\ell\right)$
if $s=\ell$, and $\ell\geq s+1$, $\left(\mathbf{j_{i}}^{s}\right)_{\consi\left(\ell\right)}$
is defined inductively:
\[
\left(\mathbf{j_{i}}^{s}\right)_{\consi\left(\ell\right)}=\left(\mathbf{j_{i}}^{s}\right)_{\consi\left(\mathbf{k}\left(\ell\right)\right)}.
\]
\end{enumerate}
\end{defn}

Thanks to assumption \ref{assu:Piecewise_assumption_Domain}, for
any $i\in\left\{ 1,\cdots,N+1\right\} $, we can always find a construction
map $\consi$ with $\consi\left(1\right)=i$. A simple example is:
\begin{example}
\label{exa:5_pieces_example}Let $\Omega=\Omega_{1}\cup\Omega_{2}\cup\Omega_{3}\cup\Omega_{4}\cup\Omega_{5}$
in \figref{5inclusionstart}. Then $\consi_{1}:\left\{ 1,2,3,4,5\right\} \rightarrow\left\{ 2,3,1,5,4\right\} $
and $\mathbf{i_{2}}:\left\{ 1,2,3,4,5\right\} \rightarrow\left\{ 2,1,5,4,3\right\} $
are two different construction maps with $\consi_{1}\left(1\right)=\consi_{2}\left(1\right)=2$.
\end{example}

We have
\begin{rem*}
\[
\mathbf{j_{i_{1}}=\left\{ \begin{array}{c}
\left\{ 2,2,2,2,2\right\} \\
\left\{ 2,2,3,2,2\right\} \\
\left\{ 1,2,3,1,1\right\} \\
\left\{ 1,2,3,5,5\right\} \\
\left\{ 1,2,3,4,5\right\} 
\end{array}\right\} }\text{ and }\mathbf{j_{i_{2}}}=\left\{ \begin{array}{c}
\left\{ 2,2,2,2,2\right\} \\
\left\{ 1,2,2,1,1\right\} \\
\left\{ 1,2,2,5,5\right\} \\
\left\{ 1,2,2,4,5\right\} \\
\left\{ 1,2,3,4,5\right\} 
\end{array}\right\} .
\]

Note that for any construction map $\consi$, there holds $\mathbf{j}_{\consi}^{N+1}=\left\{ 1,\cdots,N+1\right\} $.
\end{rem*}
\begin{figure}[H]
\begin{centering}
\tikzset{declare function={sinhx(\x,\y)  =              sinh(\x*cos(deg(\y)))*cos(deg(\x*sin(deg(\y));}} 
\tikzset{declare function={sinhy(\x,\y)  =                cosh(\x*cos(deg(\y)))*sin(deg(\x*sin(deg(\y));}} 
\begin{tikzpicture}
\node at (0, -1.8) {};
    \filldraw[domain=0:2*pi,samples=200,  fill=pink!30!white, variable=\t]  plot ({-0.6+3*exp(-0.5*cos(deg(\t)))*cos(deg(\t))},{-0.05+1.6*sin(deg(\t))});        \node[text=black]  at ( -1.4,0.0) {$5$};              \tikzmath{\x = 1;}        
\filldraw[domain=0:2*pi,samples=200,  fill=cyan, variable=\t]  plot ({0.5*sinhx(\x,\t)+0.86*sinhy(\x,\t)},{0.5*sinhy(\x,\t)-0.86*sinhx(\x,\t)});        \node  at ( 0.8,0) {$1$};        \tikzmath{\x = 0.6;}        
\filldraw[domain=0:2*pi,samples=200,  fill=blue, variable=\t]  plot ({\x*cos(deg(\t))},{\x*sin(deg(\t))});        \node[text=white]  at ( 0.3,0) {$2$};        \tikzmath{\x = 0.2;}        \filldraw[domain=0:2*pi,samples=200,  fill=red, variable=\t]  plot ({-0.1 +\x*cos(deg(\t))},{\x*sin(deg(\t))});        \node[text=white] at (-0.1,0) {$3$};         \tikzmath{\x = 0.8;}        \filldraw[domain=0:2*pi,samples=200,  fill=green, variable=\t] plot ({-2+ exp(-cos(deg(\t)))*cos(deg(\t))},{-0.1+exp(-0.5*cos(deg(\t)))*sin(deg(\t))});       \node at (-3,0) {$4$};
\end{tikzpicture}
\par\end{centering}
\caption{\label{fig:5inclusionstart}A 4 inclusion configuration.}
\end{figure}
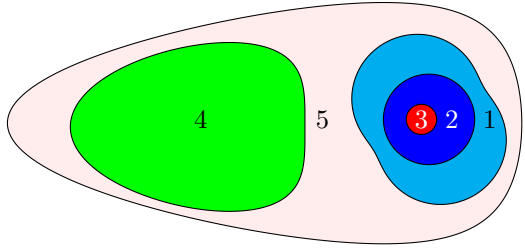

In the sequel, it would be convenient to assume that the Dirichlet
problem boundary value problem associated to $L$ is well posed in
$\Omega$, as it would allow us to control the norm of solutions by
their boundary traces. In fact, well-posedness for a large family
of sub-problems will be used. We denote $L\left[i_{1},i_{2},\cdots,i_{N+1}\right]$
for $i_{1},\cdots,i_{N+1}\in\left\{ 1,\cdots,N+1\right\} $, the second
order elliptic operator with given coefficients $A_{i_{j}},b_{i_{j}},c_{i_{j}},q_{i_{j}}$in
$\Omega_{j}$. We shall use the following lemma.
\begin{lem}
\label{lem:translation_to_get_well-posedness}There exists some $\vartheta>0$
such that for any $\translation\in\left(0,\translationconstant\right)$,
all Dirichlet boundary value problems associated with $L\left[i_{1},\cdots,i_{N+1}\right]+\translation$
where $i_{1},\cdots,i_{N+1}\in\left\{ 1,\cdots,N+1\right\} $ are
well-posed in $\Omega$.
\end{lem}

The proof of this lemma is given in \subsecref{ProofLemTrans}. 
For any $\translation\in\left(0,\vartheta\right)$ fixed, we first
prove the \thmref{piecewise_regular-construction} for $L+\translation$.
To simplify notations, we write $L$ for $L+\translation$ . Thanks
to \lemref{translation_to_get_well-posedness}, the Dirichlet boundary
value problem associated with $L\left[i_{1},\cdots,i_{N+1}\right]$
is well-posed for any $i_{1},\cdots,i_{N+1}\in\left\{ 1,\cdots,N+1\right\} $.
In the last step, we shall revert to the original operator, now $L-\translation$,
to prove the \thmref{piecewise_regular-construction} using the smallness
of $\translation$ and the regularity of the coefficients.

The proof of \thmref{piecewise_regular-construction} relies on a
series of lemmas. To start the construction, we exhibit functions
satisfying the requirement (\ref{eq:requirement}), which satisfy
$Lu_{i}^{x}=0$ in a neighbourhood of $x$.
\begin{lem}
\label{lem:localcase} Given $j\in\left\{ 1,\cdots,N+1\right\} ,$
for any $\sigma\in\left(0,1\right)$, there exists $\epsilon\in\left(0,1\right)$
depending on $\lambda,\sigma$,$d$ and $\translation$ given by \lemref{translation_to_get_well-posedness}
only such that for any point $x\in\Omega\setminus\cup_{i,j}\Gamma_{ij}$,
there exist $u_{1}^{x},u_{2}^{x},\cdots,u_{d+1}^{x}\in\left(H^{1}\left(\Omega\right)\right)^{d+1}$
such that for $i\in\{1,2,\cdots,d+1\},$ $L\left[j,\cdots,j\right]u_{i}^{x}=0$
in $\Omega$. Moreover there holds
\[
\det\JacJ\left(u_{1}^{x},u_{2}^{x},\cdots,u_{d}^{x}\right)(y)>\sigma
\]
\end{lem}

for any $y\in B\left(x,\epsilon\right)\cap\Omega$.
\begin{proof}
Fix $j=1$. Consider $x=0\in\Omega$, and $B_{x}=B\left(0,2\text{diam}\Omega\right)$
a ball centred in $x$ containing $\Omega$. In the sequel $C$ represents
any constant depending on $d$ and $\lambda$ given in assumption~\ref{assu:regularitecoeff_},
and $\translation$ given by \lemref{translation_to_get_well-posedness},
only. Note that the coefficients of $L\left[1,\cdots,1\right]$, namely
$A_{1},b_{1},c_{1},q_{1}+\translation$, are Hölder continuous on
$B_{x}$. Consider the constant coefficient partial differential operator
\begin{equation}
L_{0}:v\to-\text{div}\left(A_{1}\left(0\right)Dv+b_{1}\left(0\right)v\right)+c_{1}\left(0\right)\cdot Dv+\left(q_{1}\left(0\right)+\translation\right)v.\label{eq:constcoef for piece-wise regular elliptic}
\end{equation}
For $i=1,\cdots,d$, let $u_{i}=f\left(x_{i}\right)$ be the solution
of constant coefficients ODE 
\[
\begin{cases}
-\left(A_{1}\right)_{ii}\left(0\right)f^{\prime\prime}\left(t\right)+\left(c_{1}^{i}\left(0\right)-b_{1}^{i}\left(0\right)\right)f^{\prime}\left(t\right)+\left(q_{1}\left(0\right)+\translation\right)f\left(t\right)=0 & \text{for all \ensuremath{t\in\mathbb{R}},}\\
f^{\prime}(0)=1,\\
f(0)=0.
\end{cases}
\]
Let $u_{d+1}=f\left(x_{1}\right)$ be the solution of the following
second-order constant coefficients ODE initial value problem:
\[
\begin{cases}
-\left(A_{1}\right)_{11}\left(0\right)f^{\prime\prime}\left(t\right)+\left(c_{1}^{1}(0)-b_{1}^{1}(0)\right)f^{\prime}\left(t\right)+\left(q_{1}\left(0\right)+\translation\right)f\left(t\right)=0 & \text{for all \ensuremath{t\in\mathbb{R}},}\\
f^{\prime}(0)=0,\\
f(0)=1.
\end{cases}
\]
We observe that, for all $i\in\left\{ 1,\cdots,d+1\right\} ,$ $L_{0}u_{i}=0$
in $\Omega$, and $\det\JacJ\left(u_{1},\cdots,u_{d+1}\right)\left(0\right)=1.$
We now turn to solutions for the boundary value problem variable coefficients.
Set $V_{\epsilon}=B(0,2\epsilon)\subset B_{x}$ for some $\epsilon\in\left(0,\min\left(\frac{1}{2},\frac{1}{2}\text{diam\ensuremath{\Omega}}\right)\right)$
to be chosen later. We shall construct $u_{1}^{x},\cdots,u_{d+1}^{x}$
in $H_{\text{loc}}^{1}\left(B_{x}\right)$, the required construction
being obtained by taking the restriction to $\Omega.$ Consider the
$d+1$ Dirichlet problems
\[
\begin{cases}
L\left[1,\cdots,1\right]v_{j}=0 & \text{in }V_{\epsilon},\\
v_{j}=u_{j} & \text{on }\partial V_{\epsilon},
\end{cases}\quad j=1,\cdots,d+1.
\]
Note that this problem is well posed for $\epsilon$ small enough.
Thanks to \lemref{regula-global}, $v_{1},v_{2},\cdots,v_{d+1}$ are
well defined and in $C^{1,\alpha}\left(V_{\epsilon}\right)$. Set
for $i=1,\cdots,d+1$, $\delta_{i}^{1}\coloneqq-\left(A_{1}(x)-A_{1}(0)\right)Du_{i}-\left(b_{1}(x)-b_{1}\left(0\right)\right)u_{i},$
and $\delta_{i}^{2}:=\left(c_{1}(x)-c_{1}\left(0\right)\right)\cdot Du_{i}+\left(q_{1}\left(x\right)-q_{1}\left(0\right)\right)u_{i}.$
Then, for each $i$, 
\[
L\left[1,\cdots,1\right]\left(u_{i}-v_{i}\right)=\text{div}\left(\delta_{i}^{1}\right)+\delta_{i}^{2},
\]
and
\[
\left\Vert Du_{i}-Dv_{i}\right\Vert _{C^{0,\alpha}\big(V_{\epsilon}\big)}\leq C\left(\left\Vert \delta_{i}^{1}\right\Vert _{C^{0,\alpha}\big(V_{\epsilon}\big)}+\left\Vert \delta_{i}^{2}\right\Vert _{C^{0,\alpha}\big(V_{\epsilon}\big)}\right).
\]
 In particular, 
\begin{equation}
\left\Vert Du_{i}-Dv_{i}\right\Vert _{C^{0,\frac{\alpha}{2}}\big(V_{\epsilon}\big)}\leq C\left(\left\Vert \delta_{i}^{1}\right\Vert _{C^{0,\alpha}\big(V_{\epsilon}\big)}+\left\Vert \delta_{i}^{2}\right\Vert _{C^{0,\alpha}\big(V_{\epsilon}\big)}\right)\text{diam}\left(V_{\epsilon}\right)^{\frac{\alpha}{2}}\leq C\epsilon^{\frac{\alpha}{2}}.\label{eq:lmv}
\end{equation}
By an integration by part, and Poincaré's inequality (with a constant
chosen to be valid for any $\epsilon\in\left(0,1\right)$),
\[
\left\Vert Du_{i}-Dv_{i}\right\Vert _{L^{2}\big(V_{\epsilon}\big)}\leq C\left(\left\Vert \delta_{i}^{1}\right\Vert _{L^{2}\big(V_{\epsilon}\big)}+C_{\text{Poincaré}}\left\Vert \delta_{i}^{2}\right\Vert _{L^{2}\big(V_{\epsilon}\big)}\right).
\]
We compute, using the Hölder regularity of the parameters, 
\[
\int_{V_{\epsilon}}\left(\delta_{i}^{1}\right)^{2}\text{d}x+\int_{V_{\epsilon}}\left(\delta_{i}^{2}\right)^{2}\text{d}x\leq C\epsilon^{d+2\alpha}.
\]
Inserting this estimate in (\ref{eq:lmv}) we obtain
\[
\left\Vert u_{i}-v_{i}\right\Vert _{H_{0}^{1}\big(V_{\epsilon}\big)}\leq C\epsilon^{d/2+\alpha},
\]
and using that for any $x,y\in V_{\epsilon}$ and any $f,$there holds
$\left\Vert f\right\Vert _{\infty}\leq\left\Vert f\right\Vert _{C^{0,\frac{\alpha}{2}}\big(V_{\epsilon}\big)}\text{diam}\left(V_{\epsilon}\right)^{\frac{\alpha}{2}}+\left|V_{\epsilon}\right|^{-\frac{1}{2}}\left\Vert f\right\Vert _{L^{2}\big(V_{\epsilon}\big)},$we
conclude that 
\[
\left\Vert Du_{i}-Dv_{i}\right\Vert _{L^{\infty}\left(V_{\epsilon}\right)}\leq C\epsilon^{\alpha}.
\]
Because of assumption~\assuref{regularitecoeff_}, the operator $L\left[1,\cdots,1\right]$
enjoys a Unique Continuation Property on a $B_{x}$. Thus, for each
$i$ there exists $u_{i}^{x}\in H^{1}\left(B_{x}\right)\cap C_{\text{loc}}^{1,\alpha}\left(B_{x}\right)$
such that $L\left[1,\cdots,1\right]u_{i}^{x}=0\text{ on }B_{x},$
and $\left\Vert u_{i}^{x}-v_{i}\right\Vert _{L^{2}\left(V_{\epsilon}\right)}<\epsilon.$
Thanks to \lemref{regula-global} this implies $\left\Vert Du_{i}^{x}-Dv_{i}\right\Vert _{L^{\infty}\left(B_{\epsilon}\right)}\leq C\epsilon^{\alpha},$
where $B_{\epsilon}=B\left(0,\epsilon\right)$, and in turn, 
\[
\left\Vert Du_{i}-Du_{i}^{x}\right\Vert _{L^{\infty}\left(B_{\epsilon}\right)}\leq C\epsilon^{\alpha}.
\]
Since $\det\JacJ$ is multi-linear, 
\begin{multline*}
\det\left(\JacJ\left(u_{1},\cdots,u_{d+1}\right)-\JacJ\left(u_{1}^{x},\cdots,u_{d+1}^{x}\right)\right)\\
\leq\left(d+1\right)\left(\sum_{i=1}^{d+1}\left|Du_{i}\right|+\left|Du_{i}^{x}\right|\right)^{d}\max_{1\leq i\leq d+1}\left|Du_{i}-Du_{i}^{x}\right|
\end{multline*}
Therefore 
\[
\sup_{B_{\epsilon}}\det\left|\JacJ\left(u_{1},\cdots,u_{d+1}\right)-\JacJ\left(u_{1}^{x},\cdots,u_{d+1}^{x}\right)\right|\leq C\epsilon^{\alpha}.
\]
Since $\det\JacJ\left(u_{1},u_{2},\cdots,u_{d+1}\right)\left(0\right)=1$,
for any $\sigma\in\left(0,1\right)$ there exists $\epsilon$, depending
$\lambda,d,\sigma$ and $\translation$ only such that 
\[
\min_{B_{\epsilon}}\det\JacJ\left(u_{1}^{x},\cdots,u_{d+1}^{x}\right)>\sigma.
\]
\end{proof}
The following lemma extends a solution across an interface.
\begin{lem}
\label{lem:SmallExtension} Let $\consi$ be a construction map as
defined in definition~\ref{def:construction-map}, and $\mathbf{j_{i}}$
the associated index map. Given $k\in\left\{ 1,\ldots,N+1\right\} $,
write $\Gamma_{k}=\partial\Omega_{I_{k-1}}\cap\partial\Omega_{\consi(k)},$
and $L_{k}=L\left[\mathbf{j_{i}}^{k}\right]$. Let 
\[
W_{k}=\mathring{\overline{\cup_{\left\{ \ell:\left(\mathbf{j_{i}}^{k}\right)_{\ell}=\left(\mathbf{j_{i}}^{k-1}\right)_{\ell}\right\} }\Omega_{\ell}}}.
\]
In other words, $W_{k}$ is the open set where coefficients of $L_{k}$
are almost everywhere the same as those of $L_{k-1}.$ Suppose that
$u\in H^{1}\left(W_{k}\right)$ is a weak solution of $L_{k}u=L_{k-1}u=0$
in $W_{k}$. For any $\delta>0,$ there exists an open set $U$, such
that $W_{k}\cup\Gamma_{k}\subset U\subset\Omega$ and $v\in H^{1}(U)$
such that $L_{k}v=0$ in $U$ and 
\[
\left\Vert u-v\right\Vert _{H^{1}\left(W_{k}\right)}<\delta.
\]
\end{lem}

\begin{proof}
Suppose that $\Gamma_{k}=\Gamma_{1\consi(k)}$, that is, the subdomain
within $\Omega_{I_{k}}$ for whom $\Gamma_{k}$ is a connected component
of its boundary is $\Omega_{1}$. Write $\consi(k)=k$, and $\omega_{1}=\Omega_{1}\cap\left\{ x:d\left(x,\Gamma_{k}\right)<d_{0}\right\} $,
where $d_{0}$ is given by \eqref{defkappa}. Thanks to assumption~\ref{assu:Piecewise_assumption_Domain},
there exists a $C^{2}$ diffeomorphism $\psi_{k}:U_{1,k}\rightarrow V_{1,k},\text{\ensuremath{\Gamma_{k}\mapsto\partial B_{1}}}$,
where $U_{1,k}$ and $V_{1,k}$ are neighbourhoods of $\Gamma_{k}$
and $\partial B_{1}$. Take $\eta>0$ small enough such that $\psi_{k}^{-1}\left(\partial B_{1-\eta}\right)\subset U_{1,k}\cap\omega_{1}.$
Take $t\in\left(\frac{1}{2},1\right)$ and set 
\[
U^{t}\coloneqq\left\{ \psi_{k}^{-1}x\,:\,tx\in\partial B_{1}\setminus\partial B_{1-\eta}\right\} =\left\{ \psi_{k}^{-1}x\,:\,x\in\partial B_{\frac{1}{t}}\setminus\partial B_{\frac{1}{t}\left(1-\eta\right)}\right\} \subset U_{1,k}.
\]
 An example of such a construction is illustrated in \figref{small-ext-pic}.
In what follows, $C$ is any constant, which may change from line
to line, depending on $\Omega$, $\Omega_{1}$, $d$, $\lambda$,
$\translation$ and $\left\Vert \psi_{k}^{-1}\right\Vert _{C^{2}}$.
Write $Y\coloneqq U^{t}\cap\Omega_{1}$, $G\coloneqq U^{t}\cap\Omega_{\consi\left(k\right)}$.
Define 
\[
u^{t}(x)=u\left(\psi_{k}^{-1}\left(t\psi_{k}\left(x\right)\right)\right)\in H^{1}\left(U^{t}\right).
\]

There exists some $\eta_{0}>0$, such that for any $0<\eta<\eta_{0},$
and any $t\in\left(\frac{1}{2},1\right)$ there holds
\begin{equation}
\forall u\in H_{0}^{1}\left(U^{t}\right),\left\langle L_{k}u,u\right\rangle _{H^{-1}\left(U^{t}\right),H_{0}^{1}\left(U^{t}\right)}\geq\frac{1}{3}\lambda\left\Vert u\right\Vert _{H_{0}^{1}\left(U^{t}\right)}^{2}.\label{eq:coerciveL}
\end{equation}

We establish this claim in \subsecref{Proof-of-Poincare}. 

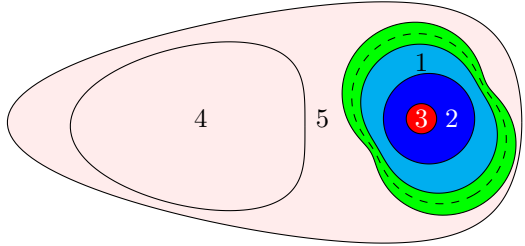
\begin{figure}[H]
\centering{} \tikzset{declare function={sinhx(\x,\y)  =                sinh(\x*cos(deg(\y)))*cos(deg(\x*sin(deg(\y));}} \tikzset{declare function={sinhy(\x,\y)  =                cosh(\x*cos(deg(\y)))*sin(deg(\x*sin(deg(\y));}}    
\begin{tikzpicture}  
\node at (0, -1.8) {};
\filldraw[domain=0:2*pi,samples=200,  fill=pink!30!white, variable=\t]  plot ({-0.6+3*exp(-0.5*cos(deg(\t)))*cos(deg(\t))},{-0.05+1.6*sin(deg(\t))});        \node[text=black]  at ( -1.4,0.0) {$5$};               \tikzmath{\x = 1.1;}        \filldraw[domain=0:2*pi,samples=200,  fill=green, variable=\t]  plot ({0.5*sinhx(\x,\t)+0.86*sinhy(\x,\t)},{0.5*sinhy(\x,\t)-0.86*sinhx(\x,\t)});       \tikzmath{\x = 1;}        \draw[domain=0:2*pi,samples=200,   dashed, variable=\t]  plot ({0.5*sinhx(\x,\t)+0.86*sinhy(\x,\t)},{0.5*sinhy(\x,\t)-0.86*sinhx(\x,\t)});        \tikzmath{\x = 0.9;}        \filldraw[domain=0:2*pi,samples=200,  fill=cyan, variable=\t]  plot ({0.5*sinhx(\x,\t)+0.86*sinhy(\x,\t)},{0.5*sinhy(\x,\t)-0.86*sinhx(\x,\t)});        \node  at ( -0.1,0.75) {$1$};        \tikzmath{\x = 0.6;}        \filldraw[domain=0:2*pi,samples=200,  fill=blue, variable=\t]  plot ({\x*cos(deg(\t))},{\x*sin(deg(\t))});        \node [text=white]  at ( 0.3,0) {$2$};        \tikzmath{\x = 0.2;}        \filldraw[domain=0:2*pi,samples=200,  fill=red, variable=\t]  plot ({-0.1 +\x*cos(deg(\t))},{\x*sin(deg(\t))});        \node[text=white] at (-0.1,0) {$3$};         \tikzmath{\x = 0.8;}        \filldraw[domain=0:2*pi,samples=200,  fill=pink!30!white, variable=\t] plot ({-2+ exp(-cos(deg(\t)))*cos(deg(\t))},{-0.1+exp(-0.5*cos(deg(\t)))*sin(deg(\t))});       \node at (-3,0) {$4$};\end{tikzpicture}\caption{\label{fig:small-ext-pic}For the example shown in \figref{5inclusionstart},
this illustrates an extension across $\Gamma_{15}$, the dashed line.
The green area correspond $U_{1,5}^{t},$ with $t=\frac{1}{10}$.
In this example, the Dirichlet problem is $L\left[1,2,3,5,5\right]u=0$. }
\end{figure}
Consider the Dirichlet boundary value problem in $U^{t}$
\[
\begin{cases}
Lv=0 & \text{in \ensuremath{U^{t}}},\\
v=u^{t} & \text{on \ensuremath{\partial U^{t},}}
\end{cases}
\]
which is well-posed thanks to (\ref{eq:coerciveL}). We estimate
\begin{eqnarray*}
\left\langle L\left(u^{t}-v\right),u^{t}-v\right\rangle _{H^{-1}\left(U^{t}\right)\times H^{1}\left(U^{t}\right)} & = & \left\langle Lu^{t},u^{t}-v\right\rangle _{H^{-1}\left(U^{t}\right)\times H^{1}\left(U^{t}\right)}\\
 & \leq & \left|J_{0}\right|+\left|J_{1}\right|+\left|J_{2}\right|,
\end{eqnarray*}
where 
\[
J_{0}=\int_{Y}AD\left(u^{t}-u\right)\cdot D\left(u^{t}-v\right)+\left(u^{t}-u\right)\left(b+c\right)\cdot D\left(u^{t}-v\right)+\left(q+\translation\right)\left(u^{t}-u\right)\left(u^{t}-v\right)\text{d}x,
\]
\[
J_{1}=\int_{G}ADu^{t}\cdot D\left(u^{t}-v\right)+bu^{t}\cdot D\left(u^{t}-v\right)+c\cdot Du^{t}\left(u^{t}-v\right)+\left(q+\translation\right)u^{t}\left(u^{t}-v\right)\text{d}x,
\]
and 
\[
J_{2}=-\int_{\Gamma_{k}}\left(ADu^{t}+bu^{t}\right)\cdot n\left(u^{t}-v\right)dS.
\]
Thanks to the Hölder regularity of $u$ and $Du$, see \lemref{regula-global},
\[
\left|u^{t}\left(x\right)-u(x)\right|=\left|u\left(\psi_{k}^{-1}\left(t\psi_{k}\left(x\right)\right)\right)-u\left(\psi_{k}^{-1}\left(\psi_{k}\left(x\right)\right)\right)\right|\leq C\left|1-t\right|^{\alpha}\left\Vert u\right\Vert _{C^{0,\alpha}\left(\omega_{1}\right)},
\]
Similarly, 
\[
\left|Du^{t}\left(x\right)-Du(x)\right|\leq C\left|1-t\right|^{\alpha}\left\Vert Du\right\Vert _{C^{0,\alpha}\left(\omega_{1}\right)},
\]
 and altogether 
\[
\left|J_{0}\right|\leq C\left|1-t\right|^{\alpha}\left\Vert u\right\Vert _{C^{1,\alpha}\left(\omega_{1}\right)}\left\Vert v-u^{t}\right\Vert _{H^{1}\left(U^{t}\right)}.
\]
To estimate $J_{1}$, we write
\[
\left|J_{1}\right|\leq C\left\Vert u^{t}\right\Vert _{H^{1}\left(G\right)}\left\Vert v-u^{t}\right\Vert _{H^{1}\left(U^{t}\right)},
\]
and by interpolation, 
\[
\left\Vert u^{t}\right\Vert _{H^{1}\left(G\right)}\leq C\left(\left\Vert u\right\Vert _{L^{\infty}\left(\Omega_{1}\right)}+\left\Vert Du\right\Vert _{L^{\infty}\left(\Omega_{1}\right)}\right)\left|G\right|^{\frac{1}{2}}\leq C\left(\left\Vert u\right\Vert _{L^{\infty}\left(\Omega_{1}\right)}+\left\Vert Du\right\Vert _{L^{\infty}\left(\Omega_{1}\right)}\right)\left(1-t\right)^{\frac{1}{2}}.
\]
Thus altogether, writing $\beta=\min\left(\alpha,\frac{1}{2}\right)$.
\begin{equation}
\left|J_{0}\right|+\left|J_{1}\right|\leq C\left\Vert u\right\Vert _{C^{1,\alpha}\left(\Omega_{1}\right)}\left\Vert v-u^{t}\right\Vert _{H^{1}\left(U^{t}\right)}\left|1-t\right|^{\beta}.\label{eq:estJ0J1}
\end{equation}
Note that
\begin{eqnarray}
J_{2} & \leq & \left\Vert \left(ADu+bu\right)\cdot n\right\Vert _{L^{2}\left(\Gamma_{k}\right)}\left\Vert u^{t}-v\right\Vert _{L^{2}\left(\Gamma_{k}\right)}\nonumber \\
 & \leq & C\left(\left\Vert u\right\Vert _{L^{\infty}\left(\Omega_{1}\right)}+\left\Vert Du\right\Vert _{L^{\infty}\left(\Omega_{1}\right)}\right)\left\Vert u^{t}-v\right\Vert _{L^{2}\left(\Gamma_{k}\right)}.\label{eq:estJ2}
\end{eqnarray}
Note that for every $x\in\Gamma_{k}$, $\psi_{k}^{-1}\left(\frac{1}{t}\psi_{k}\left(x\right)\right)\in\partial U^{t}.$
Since $v=u^{t}$ on $\partial U^{t}$, we find on $\Gamma_{k}$
\begin{eqnarray*}
\left(u^{t}-v\right)\left(x\right) & = & \left(u^{t}-v\right)\left(\psi_{k}^{-1}\circ\psi_{k}\left(x\right)\right)-\left(u^{t}-v\right)\left(\psi_{k}^{-1}\left(\frac{1}{t}\psi_{k}\left(x\right)\right)\right)\\
 & = & \int_{\frac{1}{t}}^{1}D\left(\left(u^{t}-v\right)\circ\psi_{k}^{-1}\right)\left(\theta x\right)\cdot xd\theta,
\end{eqnarray*}
Applying Cauchy--Schwarz, we find
\[
\left|\left(u^{t}-v\right)\left(x\right)\right|\leq C\left|1-t\right|^{\frac{1}{2}}\left(\int_{\frac{1}{t}}^{1}\left|D\left(\left(u^{t}-v\right)\circ\psi_{k}^{-1}\right)\left(\theta x\right)\right|^{2}d\theta\right)^{\frac{1}{2}},
\]
and integrating over $\Gamma_{k}$
\begin{eqnarray}
\left\Vert u^{t}-v\right\Vert _{L^{2}\left(\Gamma_{k}\right)}^{2} & \leq & C\left|1-t\right|\int_{\Gamma_{k}}\int_{\frac{1}{t}}^{1}\left|D\left(\left(u^{t}-v\right)\circ\psi_{k}^{-1}\right)\left(\theta x\right)\right|^{2}d\theta dx\nonumber \\
 & \leq & C\left|1-t\right|\left\Vert u^{t}-v\right\Vert _{H^{1}\left(G\right)}^{2}.\label{eq:estgammas}
\end{eqnarray}
In turn, combining (\ref{eq:estJ0J1}), (\ref{eq:estJ2}) and (\ref{eq:estgammas}),
\[
\left|\left\langle L\left(u^{t}-v\right),u^{t}-v\right\rangle \right|\leq C\left|1-t\right|^{\beta}\left\Vert u\right\Vert _{C^{1,\alpha}\left(\Omega_{1}\right)}\left\Vert u^{t}-v\right\Vert _{H^{1}\left(U^{t}\right)}.
\]
 Thanks to (\ref{eq:coerciveL}), this implies 
\[
\left\Vert u^{t}-v\right\Vert _{H^{1}\left(U^{t}\right)}\leq C\left|1-t\right|^{\beta}\left\Vert u\right\Vert _{C^{1,\alpha}\left(\Omega_{1}\right)}.
\]
 For every fixed $t$ consider the following system
\begin{equation}
\begin{cases}
L_{k}S=0 & \text{in \ensuremath{\Omega\setminus\psi_{k}^{-1}\left(\partial B_{\frac{1}{t}\left(1-\eta\right)}\right)}}\\
S=0 & \text{on \ensuremath{\partial\Omega}}\\
\left[S\right]=u-v & \text{on \ensuremath{\psi_{k}^{-1}\left(\partial B_{\frac{1}{t}\left(1-\eta\right)}\right)}}\\
\left[\left(ADS+bS\right)\cdot n\right]=\left(ADu+bu\right)\cdot n-\left(ADv+bv\right)\cdot n & \text{on \ensuremath{\psi_{k}^{-1}\left(\partial B_{\frac{1}{t}\left(1-\eta\right)}\right)}},
\end{cases}\label{eq:jumps}
\end{equation}
Where $\left[\cdot\right]$ denotes the jump across the boundary,
thanks to \lemref{translation_to_get_well-posedness}, this problem
is well posed, and there exists some $S\in H^{1}\left(\Omega\setminus\psi_{k}^{-1}\left(\partial B_{\frac{1}{\lambda}\left(1-\eta\right)}\right)\right)$
solution of \eqref{jumps}. Moreover there holds:
\begin{eqnarray*}
\left\Vert S\right\Vert _{H^{1}\left(\Omega\setminus\psi_{k}^{-1}\left(\partial B_{\frac{1}{\lambda}\left(1-\eta\right)}\right)\right)} & \leq & C\left(\left\Vert u-v\right\Vert _{H^{1/2}\left(\ensuremath{\psi_{k}^{-1}\left(\partial B_{\frac{1}{\lambda}\left(1-\eta\right)}\right)}\right)}\right.\\
 &  & \left.+\left\Vert \left(AD\left(u-v\right)+b\left(u-v\right)\right)\cdot n\right\Vert _{H^{-1/2}\left(\ensuremath{\psi_{k}^{-1}\left(\partial B_{\frac{1}{\lambda}\left(1-\eta\right)}\right)}\right)}\right)\\
 & \leq & C\left\Vert u-v\right\Vert _{H^{1}\left(Y\right)}.
\end{eqnarray*}
Using the triangle inequality, this yields,
\begin{eqnarray*}
\left\Vert S\right\Vert _{H^{1}\left(\Omega\setminus\psi_{k}^{-1}\left(\partial B_{\frac{1}{\lambda}\left(1-\eta\right)}\right)\right)} & \leq & C\left(\left\Vert u^{t}-u\right\Vert _{H^{1}\left(Y\right)}+\left\Vert u^{t}-v\right\Vert _{H^{1}\left(Y\right)}\right)\\
 & \leq & C\left(1-t\right)^{\beta}\left\Vert u\right\Vert _{C^{1,\alpha}\left(\Omega_{1}\right)}.
\end{eqnarray*}
Take $\tilde{v}_{t}=v\one_{U^{t}}+\one_{\left(W_{k}\setminus\Omega_{1}\right)\cup\psi_{k}^{-1}\left(B_{\frac{1}{t}\left(1-\eta\right)}\right)}u+S\one_{\Omega\setminus\psi_{k}^{-1}\left(\partial B_{\frac{1}{t}\left(1-\eta\right)}\right)}$.
By construction, we have $\tilde{v}_{t}\in H^{1}\left(W_{k}\cup\Gamma_{k}\cup U^{t}\right)$
and there holds $\left\Vert \tilde{v}_{t}-u\right\Vert _{H^{1}\left(W_{k}\right)}\leq C\left(1-t\right)^{\beta}\left\Vert u\right\Vert _{C^{1,\alpha}\left(\Omega_{1}\right)}$.
The conclusion follows choosing $t$ close enough to $1$, and $U=W_{k}\cup\Gamma_{k}\cup U^{t}$.
\end{proof}
The third step is to extend the solution to the whole $\Omega$. 
\begin{lem}
\label{lem:extended_Runge} With the notations of (\ref{lem:SmallExtension}),
for any $\epsilon>0$, there exists a weak solution of $v\in H^{1}\left(\Omega\right)$
of $L_{k}v=0\text{ in }\Omega$ such that $\left\Vert v-u\right\Vert _{H^{1}\left(U\right)}<\epsilon$.
\end{lem}

\begin{proof}
Note that on $V_{k}=\Omega\setminus\overline{W_{k}}$ the coefficients
of $L_{k}$ are not discontinuous, and the Unique Continuation Property
holds. As a result there exists a sequence of functions $\left(u_{n}\right)_{n\in\mathbb{N}}\in H^{1}\left(V_{k}\right)^{\mathbb{N}}$
such that 
\[
L_{k}u_{n}=0\text{ on \ensuremath{V_{k}}}
\]

and

\[
\left\Vert u_{n}-u\right\Vert _{H^{1}\left(U\cap V_{k}\right)}\leq\frac{1}{n}.
\]

which implies that 
\[
\left\Vert u_{n}-u\right\Vert _{H^{1/2}\left(\partial W_{k}\right)}\leq\left\Vert u_{n}-u\right\Vert _{H^{1/2}\left(\partial\left(U\cap V_{k}\right)\right)}\leq\left\Vert u_{n}-u\right\Vert _{H^{1}\left(U\cap V_{k}\right)}\leq\frac{C}{n}.
\]
Let $\nu$ be the outer normal vector of $\partial W_{k}$ and let
\begin{eqnarray*}
\mathcal{F}_{1}:H^{1}\left(W_{k}\right) & \rightarrow & H^{-1/2}\left(\partial W_{k}\right)\\
u & \mapsto & \left(\tilde{A}|_{W_{k}}Du+\tilde{b}|_{W_{k}}u\right)\cdot\nu
\end{eqnarray*}

and 
\begin{eqnarray*}
\mathcal{F}_{2}:H^{1}\left(U\setminus W_{k}\right) & \rightarrow & H^{-1/2}\left(\partial W_{k}\right)\\
u & \mapsto & \left(\tilde{A}|_{V_{k}}Du+\tilde{b}|_{V_{k}}u\right)\cdot\nu
\end{eqnarray*}
Since $u\in H^{1}\left(U\right)$ is a weak solution of $L_{k}u=0$
in $U$, there holds $\mathcal{F}_{1}\left(u\right)=\mathcal{F}_{2}\left(u\right)$
on $\partial\Omega_{I_{k}}$. As a result,
\begin{eqnarray*}
\left\Vert \mathcal{F}_{1}\left(u\right)-\mathcal{F}_{2}\left(u_{n}\right)\right\Vert _{H^{-1/2}\left(\partial W_{k}\right)} & \leq & \left\Vert \mathcal{F}_{2}\left(u\right)-\mathcal{F}_{2}\left(u_{n}\right)\right\Vert _{H^{-1/2}\left(\partial W_{k}\right)}\\
 & \leq & \left\Vert u_{n}-u\right\Vert _{H^{1}\left(U\cap V_{k}\right)}\\
 & \leq & \frac{C}{n}.
\end{eqnarray*}
Consider the following system in $\Omega$
\begin{equation}
\begin{cases}
L_{k}s_{n}=0 & \text{in \ensuremath{\Omega\setminus\partial W_{k}}}\\
s_{n}=0 & \text{on \ensuremath{\partial\Omega}}\\
\left[s_{n}\right]=u-u_{n} & \text{on \ensuremath{\partial}\ensuremath{\ensuremath{W_{k}}}}\\
\left[\left(ADs_{n}+bs_{n}\right)\cdot\nu\right]=\mathcal{F}_{1}\left(u\right)-\mathcal{F}_{2}\left(u_{n}\right) & \text{on }\partial W_{k}.
\end{cases}\label{eq:JumpSn}
\end{equation}
 \lemref{translation_to_get_well-posedness} implies that there exists
$s_{n}\in H^{1}\left(\Omega\setminus\partial\ensuremath{W_{k}}\right)$,
a weak solution of \eqref{JumpSn} and there holds
\begin{equation}
\left\Vert s_{n}\right\Vert _{H^{1}\left(\Omega\setminus\partial W_{k}\right)}\leq C\left(\left\Vert \mathcal{F}_{1}\left(u\right)-\mathcal{F}_{2}\left(u_{n}\right)\right\Vert _{H^{-1/2}\left(\partial W_{k}\right)}+\left\Vert u-u_{n}\right\Vert _{H^{1/2}\left(\partial W_{k}\right)}\right)\leq\frac{C}{n}.\label{eq:estimsn}
\end{equation}
Let $v_{n}=s_{n}\mathbf{1}_{\Omega\setminus\partial W_{k}}+u\mathbf{1}_{U}+\mathbf{1}_{V_{k}}u_{n}$.
By construction, $v_{n}\in H^{1}\left(\Omega\right)$ is a weak solution
of \eqref{PDE}. Moreover, we have 
\[
\left\Vert v_{n}-u\right\Vert _{H^{1}\left(U\right)}\leq\left\Vert s_{n}\right\Vert _{H^{1}\left(\Omega\setminus\partial W_{k}\right)}+\left\Vert u_{n}-u\right\Vert _{H^{1}\left(U\setminus W_{k}\right)}\leq\frac{C}{n},
\]
and the conclusion follows.
\end{proof}
We now turn to the proof of \thmref{piecewise_regular-construction}.
\begin{proof}[Proof of \thmref{piecewise_regular-construction}]
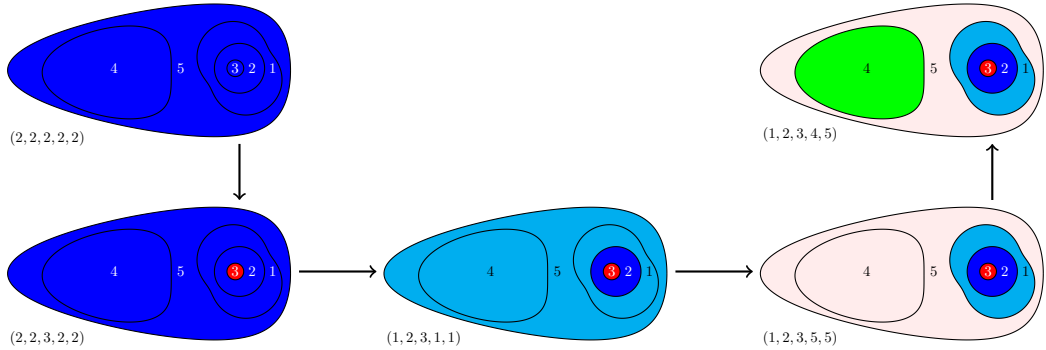
\begin{figure}[H]
\noindent \centering{}\tikzset{declare function={sinhx(\x,\y)  =                sinh(\x*cos(deg(\y)))*cos(deg(\x*sin(deg(\y));}} \tikzset{declare function={sinhy(\x,\y)  =                cosh(\x*cos(deg(\y)))*sin(deg(\x*sin(deg(\y));}} \makeatletter \tikzset{   prefix node name/.code={     \tikzset{       name/.code={\edef\tikz@fig@name{#1 ##1}}     }   } } \makeatother
       \begin{tikzpicture}[scale=0.55, every node/.style={scale=0.55}]        \begin{scope}[shift={(-9cm,5cm)}, prefix node name=TW]          \node (east) at ( 1.3, 0) {};         \node (west) at (-5.6, 0) {};         \node (north) at (0, 1.6) {};         \node (south) at (0,-1.7) {};         \node (ts) at (-4.6,-1.7) {$(2,2,2,2,2)$};         \node (tn) at (-4.6, 1.6) {};         \filldraw[domain=0:2*pi,samples=200,  fill=blue, variable=\t]  plot ({-0.6+3*exp(-0.5*cos(deg(\t)))*cos(deg(\t))},{-0.05+1.6*sin(deg(\t))});         \node[text=white]  at ( -1.4,0.0) {$5$};               \tikzmath{\x = 1;}         \filldraw[domain=0:2*pi,samples=200,  fill=blue, variable=\t]  plot ({0.5*sinhx(\x,\t)+0.86*sinhy(\x,\t)},{0.5*sinhy(\x,\t)-0.86*sinhx(\x,\t)});         \node[text=white]  at ( 0.8,0) {$1$};         \tikzmath{\x = 0.6;}         \filldraw[domain=0:2*pi,samples=200,  fill=blue, variable=\t]  plot ({\x*cos(deg(\t))},{\x*sin(deg(\t))});         \node[text=white]  at ( 0.3,0) {$2$};         \tikzmath{\x = 0.2;}         \filldraw[domain=0:2*pi,samples=200,  fill=blue, variable=\t]  plot ({-0.1 +\x*cos(deg(\t))},{\x*sin(deg(\t))});         \node[text=white] at (-0.1,0) {$3$};         \tikzmath{\x = 0.8;}         \filldraw[domain=0:2*pi,samples=200,  fill=blue, variable=\t] plot ({-2+ exp(-cos(deg(\t)))*cos(deg(\t))},{-0.1+exp(-0.5*cos(deg(\t)))*sin(deg(\t))});         \node[text=white]  at (-3,0) {$4$};        \end{scope}       \begin{scope}[shift={(-9cm,0.1cm)}, prefix node name=BW]          \node (east) at ( 1.3, 0) {};         \node (west) at (-5.6, 0) {};         \node (north) at (0, 1.6) {};         \node (south) at (0,-1.7) {};         \node (ts) at (-4.6,-1.6) {$(2,2,3,2,2)$};         \node (tn) at (-4.6, 1.6) {};         \filldraw[domain=0:2*pi,samples=200,  fill=blue, variable=\t]  plot ({-0.6+3*exp(-0.5*cos(deg(\t)))*cos(deg(\t))},{-0.05+1.6*sin(deg(\t))});         \node[text=white]   at ( -1.4,0.0) {$5$};               \tikzmath{\x = 1;}         \filldraw[domain=0:2*pi,samples=200,  fill=blue, variable=\t]  plot ({0.5*sinhx(\x,\t)+0.86*sinhy(\x,\t)},{0.5*sinhy(\x,\t)-0.86*sinhx(\x,\t)});         \node[text=white]   at ( 0.8,0) {$1$};         \tikzmath{\x = 0.6;}         \filldraw[domain=0:2*pi,samples=200,  fill=blue, variable=\t]  plot ({\x*cos(deg(\t))},{\x*sin(deg(\t))});         \node[text=white]  at ( 0.3,0) {$2$};         \tikzmath{\x = 0.2;}         \filldraw[domain=0:2*pi,samples=200,  fill=red, variable=\t]  plot ({-0.1 +\x*cos(deg(\t))},{\x*sin(deg(\t))});         \node[text=white]  at (-0.1,0) {$3$};         \tikzmath{\x = 0.8;}         \filldraw[domain=0:2*pi,samples=200,  fill=blue, variable=\t] plot ({-2+ exp(-cos(deg(\t)))*cos(deg(\t))},{-0.1+exp(-0.5*cos(deg(\t)))*sin(deg(\t))});         \node[text=white]  at (-3,0) {$4$};      \end{scope}      \begin{scope}[shift={(0cm,0.1cm)}, prefix node name=BG]          \node (east) at ( 1.3, 0) {};         \node (west) at (-5.6, 0) {};         \node (north) at (0, 1.6) {};         \node (south) at (0,-1.7) {};         \node (ts) at (-4.6,-1.6) {$(1,2,3,1,1)$};         \node (tn) at (-4.6, 1.6) {};         \filldraw[domain=0:2*pi,samples=200,  fill=cyan, variable=\t]  plot ({-0.6+3*exp(-0.5*cos(deg(\t)))*cos(deg(\t))},{-0.05+1.6*sin(deg(\t))});         \node[text=black]  at ( -1.4,0.0) {$5$};               \tikzmath{\x = 1;}         \filldraw[domain=0:2*pi,samples=200,  fill=cyan, variable=\t]  plot ({0.5*sinhx(\x,\t)+0.86*sinhy(\x,\t)},{0.5*sinhy(\x,\t)-0.86*sinhx(\x,\t)});         \node[text=black]  at ( 0.8,0) {$1$};         \tikzmath{\x = 0.6;}         \filldraw[domain=0:2*pi,samples=200,  fill=blue, variable=\t]  plot ({\x*cos(deg(\t))},{\x*sin(deg(\t))});         \node[text=white]  at ( 0.3,0) {$2$};         \tikzmath{\x = 0.2;}         \filldraw[domain=0:2*pi,samples=200,  fill=red , variable=\t]  plot ({-0.1 +\x*cos(deg(\t))},{\x*sin(deg(\t))});         \node[text=white]  at (-0.1,0) {$3$};         \tikzmath{\x = 0.8;}         \filldraw[domain=0:2*pi,samples=200,  fill=cyan, variable=\t] plot ({-2+ exp(-cos(deg(\t)))*cos(deg(\t))},{-0.1+exp(-0.5*cos(deg(\t)))*sin(deg(\t))});         \node[text=black] at (-3,0) {$4$};      \end{scope}      \begin{scope}[shift={(9cm,0.1cm)}, prefix node name=BE]          \node (east) at ( 1.3, 0) {};         \node (west) at (-5.6, 0) {};         \node (north) at (0, 1.6) {};         \node (south) at (0,-1.7) {};         \node (ts) at (-4.6,-1.6) {$(1,2,3,5,5)$};         \node (tn) at (-4.6, 1.6) {};         \filldraw[domain=0:2*pi,samples=200,  fill=pink!30!white, variable=\t]  plot ({-0.6+3*exp(-0.5*cos(deg(\t)))*cos(deg(\t))},{-0.05+1.6*sin(deg(\t))});         \node[text=black]  at ( -1.4,0.0) {$5$};               \tikzmath{\x = 1;}         \filldraw[domain=0:2*pi,samples=200,  fill=cyan, variable=\t]  plot ({0.5*sinhx(\x,\t)+0.86*sinhy(\x,\t)},{0.5*sinhy(\x,\t)-0.86*sinhx(\x,\t)});         \node[text=black]  at ( 0.8,0) {$1$};         \tikzmath{\x = 0.6;}         \filldraw[domain=0:2*pi,samples=200,  fill=blue, variable=\t]  plot ({\x*cos(deg(\t))},{\x*sin(deg(\t))});         \node[text=white]  at ( 0.3,0) {$2$};         \tikzmath{\x = 0.2;}         \filldraw[domain=0:2*pi,samples=200,  fill=red, variable=\t]  plot ({-0.1 +\x*cos(deg(\t))},{\x*sin(deg(\t))});         \node[text=white] at (-0.1,0) {$3$};         \tikzmath{\x = 0.8;}         \filldraw[domain=0:2*pi,samples=200,  fill=pink!30!white, variable=\t] plot ({-2+ exp(-cos(deg(\t)))*cos(deg(\t))},{-0.1+exp(-0.5*cos(deg(\t)))*sin(deg(\t))});         \node[text=black] at (-3,0) {$4$};      \end{scope}      \begin{scope}[shift={(9cm,5cm)}, prefix node name=TE]          \node (east) at ( 1.3, 0) {};         \node (west) at (-5.6, 0) {};         \node (north) at (0, 1.6) {};         \node (south) at (0,-1.7) {};         \node (ts) at (-4.6,-1.6) {$(1,2,3,4,5)$};         \node (tn) at (-4.6, 1.6) {};         \filldraw[domain=0:2*pi,samples=200,  fill=pink!30!white, variable=\t]  plot ({-0.6+3*exp(-0.5*cos(deg(\t)))*cos(deg(\t))},{-0.05+1.6*sin(deg(\t))});         \node[text=black]  at ( -1.4,0.0) {$5$};               \tikzmath{\x = 1;}         \filldraw[domain=0:2*pi,samples=200,  fill=cyan, variable=\t]  plot ({0.5*sinhx(\x,\t)+0.86*sinhy(\x,\t)},{0.5*sinhy(\x,\t)-0.86*sinhx(\x,\t)});         \node[text=black]  at ( 0.8,0) {$1$};         \tikzmath{\x = 0.6;}         \filldraw[domain=0:2*pi,samples=200,  fill=blue, variable=\t]  plot ({\x*cos(deg(\t))},{\x*sin(deg(\t))});         \node[text=white]  at ( 0.3,0) {$2$};         \tikzmath{\x = 0.2;}         \filldraw[domain=0:2*pi,samples=200,  fill=red, variable=\t]  plot ({-0.1 +\x*cos(deg(\t))},{\x*sin(deg(\t))});         \node[text=white] at (-0.1,0) {$3$};         \tikzmath{\x = 0.8;}         \filldraw[domain=0:2*pi,samples=200,  fill=green, variable=\t] plot ({-2+ exp(-cos(deg(\t)))*cos(deg(\t))},{-0.1+exp(-0.5*cos(deg(\t)))*sin(deg(\t))});         \node[text=black] at (-3,0) {$4$};      \end{scope}        \draw[thick,->] (TW south) -- node[right,midway] {} (BW north) ;        \draw[thick,->] (BW east)  -- node[above,midway] {} (BG west)  ;        \draw[thick,->] (BG east)  -- node[above,midway] {} (BE west)  ;        \draw[thick,->] (BE north) -- node[midway,left ] {} (TE south) ;     \end{tikzpicture} \caption{A construction following the construction map $\protect\consi:\left\{ 1,2,3,4,5\right\} \rightarrow\left\{ 2,3,1,5,4\right\} $.
Every colour represents one set of regular coefficients. At each step,
all the subdomains within which the construction has not been performed
have the same parameters as the subdomain where the solution is constructed. }
\end{figure}

 Given $\sigma>0$ and $x\in\Omega\setminus\cup_{i\neq j}\Gamma_{ij}$,
we choose a construction map $\consi\in S_{N+1}$ such that the starting
point $x$ is in the first set, $x\in\Omega_{I_{1}}$. Using \lemref{localcase}
for the first step, and then applying \lemref{SmallExtension} and
\lemref{extended_Runge} inductively, with 
\begin{eqnarray*}
L_{1} & = & L\left[\mathbf{j_{i}}^{1}\right]\\
 & \vdots\\
L_{N+1} & = & L\left[\mathbf{j_{i}}^{N+1}\right]=L,
\end{eqnarray*}
the conclusion follows.
\end{proof}
We now turn to original operator (which is represented by $L_{\text{original}}=L-\translation$),
to prove \thmref{piecewise_regular-construction}.
\begin{proof}[Proof of \thmref{piecewise_regular-construction} for $L_{\text{original}}$]
 Thanks to \thmref{piecewise_regular-construction} for $L=L_{\text{original}}+\translation$
there holds
\begin{claim}
\label{claim:caseL}For any $\sigma>0$, there exists $\epsilon>0$
such that for any $x\in\Omega\setminus\cup_{i\neq j}\Gamma_{ij}$
there exists $d+1$ solutions denoted as $u_{1}^{x},u_{2}^{x},\cdots,u_{d+1}^{x}$
such that $u_{i}^{x}\in H^{1}\left(\Omega\right)$ and $Lu_{i}^{x}=0$
in $\Omega$ for $i\in\left\{ 1,2,\cdots,d+1\right\} $, and 
\[
\left|\det\JacJ\left(u_{1}^{x},u_{2}^{x},\cdots,u_{d+1}^{x}\right)\right|\left(y\right)>\sigma,\text{ for any }y\in B\left(x,\epsilon\right)\cap\Omega_{j},\,j\in\left\{ 1,\ldots,N+1\right\} .
\]
\end{claim}

If the Dirichlet boundary value problem associated with $L_{\text{original}}$
is well-posed, then for any $i\in\left\{ 1,\ldots,d+1\right\} $ and
any $x\in\Omega\setminus\cup_{i\neq j}\Gamma_{ij}$, consider the
following Dirichlet boundary value problem:
\[
\begin{cases}
L_{\text{original}}v_{i}^{x}=0 & \text{in }\Omega\\
v_{i}^{x}=u_{i}^{x} & \text{on \ensuremath{\partial\Omega}.}
\end{cases}
\]
Then, $v_{i}^{x}-u_{i}^{x}\in H_{0}^{1}\left(\Omega\right)$ satisfies
$L_{\text{original}}\left(v_{i}^{x}-u_{i}^{x}\right)=\translation u_{i}^{x}$
in $\Omega$, and thanks to the well-posedness of $L_{\text{original}}$,
we have 
\[
\left\Vert v_{i}^{x}-u_{i}^{x}\right\Vert _{H_{0}^{1}\left(\Omega\right)}\leq C\translation,
\]
where the finite constant $C$ is independent of $x$. Thanks to the
regularity of $L_{\text{original}}$ in $\Omega_{j}$, we have 
\[
\left\Vert v_{i}^{x}-u_{i}^{x}\right\Vert _{C^{1,\alpha}\left(\Omega_{j}\right)}\leq C\translation.
\]
Take $\translation$ small enough (since $\translation\in\left(0,\vartheta\right)$
is chosen arbitrarily ) and take a corresponding $\epsilon$ given
in claim~\claimref{caseL} for $L$, thanks to the multi-linearity
of $\text{det }\JacJ$, we conclude that 
\[
\left|\det\JacJ\left(v_{1}^{x},u_{2}^{x},\cdots,v_{d+1}^{x}\right)\right|\left(y\right)>\sigma
\]
for any $y\in B\left(x,\epsilon\right)\cap\Omega_{j}$, $j\in\left\{ 1,\ldots,N+1\right\} $.

If the Dirichlet boundary value problem associated with $L_{\text{original}}$
is not well-posed, the kernel of the solution map, written $\ker\left(L_{\text{original}}\right)$
to avoid introducing additional notations, is finite dimensional,
and not empty. For any $x\in\Omega\setminus\cup_{i\neq j}\Gamma_{ij}$
and any $i\in\left\{ 1,\ldots,N+1\right\} $, take 
\[
u_{i}^{x}=u_{1}+u_{2}
\]

where $u_{1}\in\ker\left(L_{\text{original}}\right)\subset H_{0}^{1}\left(\Omega\right)\subset L^{2}\left(\Omega\right),u_{2}\in\ker\left(L_{\text{original}}\right)^{\bot}\subset L^{2}\left(\Omega\right)$.
By the Fredholm alternative, there exists a unique $v_{2}\in H^{1}\left(\Omega\right)$
such that $v_{2}-u_{2}\in H_{0}^{1}\left(\Omega\right)\cap\ker\left(L_{\text{original}}\right)^{\bot}$
satisfies

\[
\begin{cases}
L_{\text{original}}\left(v_{2}-u_{2}\right)=-L_{\text{original}}u_{2}=\translation u & \text{in \ensuremath{\Omega}},\\
v_{2}-u_{2}=0 & \text{on \ensuremath{\partial\Omega}.}
\end{cases}
\]

Furthermore, $\left\Vert v_{2}-u_{2}\right\Vert _{H_{0}^{1}\left(\Omega\right)}\leq C\translation$.
Choose $v_{i}^{x}=u_{1}+v_{2},$ which satisfies$\left\Vert v_{i}^{x}-u_{i}^{x}\right\Vert _{H^{1}\left(\Omega\right)}\leq C\translation.$
Taking $\translation$ small enough, thanks to the regularity of the
coefficients in each subdomain and the multi-linearity of $\det$
$\JacJ$, we conclude that 
\[
\left|\det\JacJ\left(v_{1}^{x},u_{2}^{x},\cdots,v_{d+1}^{x}\right)\right|\left(y\right)>\sigma.
\]
\end{proof}

\section{\label{sec:Whitney}Proof Of Proposition~\ref{pro:MNT-Light}}

We recall the definition of the geometric complement of an open set
$\Omega\subset\mathbf{R}^{d}$, which is the smallest open set $\Pi\subset\mathbf{R}^{d}$
such that $\Omega\subset\Pi$ and the genus of $\Pi$ equals to zero.
\begin{defn}
\label{def:PiecesContainedIn}Given any open set $U\subset\Omega$,
we write $g_{U}=\#\left\{ j\in\left\{ 1,\cdots,N+1\right\} :\Omega_{j}\subset U\right\} $
which is the number of pieces contained in $U$. By construction,
we have $g_{\Omega}=N+1$.
\end{defn}

\begin{lem}
\label{lem:UnitBall}Set $h_{1}=\left(x_{1},\cdots,x_{d}\right)$
on $S^{d-1}$. When $d=2,4,$ or $8$, there exists $\left\{ h_{2},\cdots,h_{d}\right\} \in\left(C^{1}\left(S^{d-1};\R^{d}\right)\right)^{d-1}$
such that $\left(h_{1},h_{2},\cdots,h_{d}\right)\in SO_{d}\left(S^{d-1}\right)$
where $SO_{d}$ refers to the real unitary matrices with positive
determinant.

Otherwise, $d\geq3$ there exists $\left\{ h_{2},\cdots,h_{d+1}\right\} \in\left(C^{1}\left(S^{d-1};\R^{d}\right)\right)^{d}$such
that $\left(h_{1},h_{2},\cdots,h_{d+1}\right)\in SO_{d+1}\left(S^{d-1}\right)$.
\end{lem}

This lemma is proved in \secref{Kervaire}.

\subsection{Proof of Proposition~\ref{pro:MNT-Light} when $d=2,4$ or $8$}
\begin{proof}
Let $\Pi_{i}$ be the geometric complement of $\Omega_{i}$, where
$i\in\left\{ 1,\cdots,N\right\} $. There exists a $C^{1,1}$ diffeomorphism
$H_{i}$ : $B_{2}\stackrel{H_{i}}{\rightarrow}\Pi_{i}$, which induces
a $C^{0,1}$ bijection on the vector fields: $DH_{i}:C^{0,1}\left(B_{2};\mathbf{\R}^{d}\right)\rightarrow C^{0,1}\left(\Pi_{i};\mathbf{\R}^{d}\right)$.
It is a map which maps $SO_{d}$ to $SO_{d}$ since the degree of
$H_{i}$ is either $1$ or $-1$ and moreover it maps the tangent
vectors (respectively the normal vector) on the sphere to the tangent
vectors (respectively the normal vector) on $\partial\Pi_{i}$, which
is the outer boundary of $\Omega_{i}$. Take $B_{r_{i}}\subset B_{2}\subset B_{r_{i}^{*}},1<r_{i}<2<r_{i}^{*}$,
such that 
\[
g_{H_{i}\left(B_{r_{i}}\right)}+1=g_{\Pi_{i}}=g_{H_{i}\left(B_{r_{i}^{*}}\right)}
\]
and such that the genus of the $\Pi_{i}\setminus H_{i}\left(B_{r_{i}}\right)$
equals to the genus of $H_{i}\left(B_{r_{i}^{*}}\right)\setminus\Pi_{i}$,
and equals one. In particular any $\Omega_{j}$, $j\neq i$, contained
in $\Pi_{i}$ are contained in $H_{i}\left(B_{r_{i}}\right)$ and
$H_{i}\left(B_{r_{i}^{*}}\right)$. Applying \lemref{UnitBall} with
$R=2$, when $d=2,4,8$ , there exists $\left\{ h_{1},\cdots,h_{d}\right\} $
in $B_{2}$ a group of $C^{1}$ unit vector fields on $\partial B_{2}.$
We construct $\left\{ f_{1},\cdots,f_{d}\right\} \in C^{0,1}\left(\overline{H_{i}\left(B_{r_{i}^{*}}\right)\setminus H_{i}\left(B_{r_{i}}\right)};SO_{d}\right)$
as follows.
\begin{criterion}
\label{cri:Rules for 2,4,8}
\end{criterion}

\begin{enumerate}
\item $\left\{ f_{1},\cdots,f_{d}\right\} =\left\{ DH_{i}\left(h_{1}\right),\cdots,DH_{i}\left(h_{d}\right)\right\} $
on $\partial\Pi_{i}$.
\item On $H_{i}\left(\partial B_{r_{i}}\right)$ and $H_{i}\left(\partial B_{r_{i}^{*}}\right)$,
let $\left\{ f_{1},\cdots,f_{d}\right\} =\left\{ e_{1},\cdots,e_{d}\right\} $.
In other words, we have $\left(f_{1},\cdots,f_{d}\right)=I_{d}$ on
$H_{i}\left(\partial B_{r_{i}}\right)$ and $H_{i}\left(\partial B_{r_{i}^{*}}\right)$.
\item Since $SO_{d}$ is path connected, at each $x\in\partial B_{r_{i}}$
there exists $S\in C^{0,1}\left(\partial B_{r_{i}}\times\left[r_{i},r_{i}^{*}\right];SO_{d}\right)$
a path such that $S\left(x,r_{i}\right)=DH_{i}^{-1}\left(I_{d}\right)$,
$S\left(x,2\right)=\left(h_{1},\cdots,h_{d}\right)\left(2\frac{x}{\left\Vert x\right\Vert }\right)$
and $S\left(x,r_{i}^{*}\right)=DH_{i}^{-1}\left(I_{d}\right)$. There
holds 
\[
\left|S\left(x,r\right)-S\left(x,r^{\prime}\right)\right|\leq\frac{C(d)}{r_{i}^{*}-r_{i}}\left|r-r^{\prime}\right|,
\]
and 
\[
\left\Vert D_{x}S\left(x,r\right)\right\Vert _{\infty}\leq C(d)\left\Vert DH_{i}^{-1}\right\Vert _{\infty}\left\Vert \left(Dh_{1},\cdots,Dh_{d}\right)\right\Vert _{\infty}.
\]
 
\item For any $r\in\left(r_{i},r_{i}^{*}\right)$, set $\left(f_{1},\cdots,f_{d}\right)\left(H_{i}\left(r\frac{x}{\left\Vert x\right\Vert }\right)\right)=DH_{i}\left(S_{x}\left(r\right)\right)\coloneqq DH_{i}\left(S\left(x,r\right)\right)$.
\end{enumerate}
In the construction above, for any $x\in\overline{H_{i}\left(B_{r_{i}^{*}}\right)\setminus H_{i}\left(B_{r_{i}}\right)}$,
we have $\left(f_{1},\cdots,f_{d}\right)(x)\in SO_{d}$. Moreover
since $\left(f_{1},\cdots,f_{d}\right)$ is constructed by a composition
of Lipschitz maps, $\left\{ f_{1},\cdots,f_{d}\right\} $ is of class
$C^{0,1}$ in $\overline{H_{i}\left(B_{r_{i}^{*}}\right)\setminus H_{i}\left(B_{r_{i}}\right)}$
. Indeed, 
\begin{eqnarray}
\left\Vert f_{k}\left(H_{i}\left(x\right)\right)-f_{k}\left(H_{i}\left(y\right)\right)\right\Vert  & \leq & \left\Vert f_{k}\left(H_{i}\left(\left\Vert x\right\Vert \frac{x}{\left\Vert x\right\Vert }\right)\right)-f_{k}\left(H_{i}\left(\frac{\left\Vert y\right\Vert +\left\Vert x\right\Vert }{2}\frac{x}{\left\Vert x\right\Vert }\right)\right)\right\Vert \nonumber \\
 & + & \left\Vert f_{k}\left(H_{i}\left(\frac{\left\Vert x\right\Vert +\left\Vert y\right\Vert }{2}\frac{x}{\left\Vert x\right\Vert }\right)\right)-f_{k}\left(H_{i}\left(\frac{\left\Vert y\right\Vert +\left\Vert x\right\Vert }{2}\frac{y}{\left\Vert y\right\Vert }\right)\right)\right\Vert \nonumber \\
 & + & \left\Vert f_{k}\left(H_{i}\left(\frac{\left\Vert x\right\Vert +\left\Vert y\right\Vert }{2}\frac{y}{\left\Vert y\right\Vert }\right)\right)-f_{k}\left(H_{i}\left(\left\Vert y\right\Vert \frac{y}{\left\Vert y\right\Vert }\right)\right)\right\Vert \nonumber \\
 & = & \left\Vert DH_{i}\left(S_{x}\left(\left\Vert x\right\Vert \right)\right)-DH_{i}\left(S_{x}\left(\frac{\left\Vert y\right\Vert +\left\Vert x\right\Vert }{2}\right)\right)\right\Vert \label{eq:Lip Proof}\\
 & + & \left\Vert DH_{i}\left(S_{y}\left(\left\Vert y\right\Vert \right)\right)-DH_{i}\left(S_{y}\left(\frac{\left\Vert y\right\Vert +\left\Vert x\right\Vert }{2}\right)\right)\right\Vert \nonumber \\
 & + & \left\Vert DH_{i}\left(S_{x}\left(\frac{\left\Vert y\right\Vert +\left\Vert x\right\Vert }{2}\right)\right)-DH_{i}\left(S_{y}\left(\frac{\left\Vert y\right\Vert +\left\Vert x\right\Vert }{2}\right)\right)\right\Vert \nonumber \\
 & \leq & C\left(d\right)\left(\frac{1}{r_{i}^{*}-r_{i}}+\left\Vert DH_{i}^{-1}\right\Vert _{\infty}\left\Vert \left(Dh_{0},\cdots,Dh_{d-1}\right)\right\Vert _{\infty}\right)\left\Vert x-y\right\Vert .\nonumber 
\end{eqnarray}

Note that for each $i\in\left\{ 1,\cdots,N\right\} $, we have $\left(f_{1},\cdots,f_{d}\right)=I_{d}$
on $\partial\left(H_{i}\left(B_{r_{i}^{*}}\right)\setminus H_{i}\left(B_{r_{i}}\right)\right)$.

Set 
\begin{equation}
\left(f_{1},\cdots,f_{d}\right)=I_{d}\text{ in \ensuremath{Q\coloneqq\Omega\setminus\cup_{i=1}^{N}\overline{H_{i}\left(B_{r_{i}^{*}}\right)\setminus H_{i}\left(B_{r_{i}}\right)}} }\label{eq:constructionInOtherPart for special d}
\end{equation}
In each $\overline{H_{i}\left(B_{r_{i}^{*}}\right)\setminus H_{i}\left(B_{r_{i}}\right)}$,
$\left\{ f_{1},\cdots,f_{d}\right\} $ is of class $C^{0,1}$, continuous
on $\partial\left(H_{i}\left(B_{r_{i}^{*}}\right)\setminus H_{i}\left(B_{r_{i}}\right)\right)$
and Lipschitz continuous in $Q$ thanks to \ref{eq:constructionInOtherPart for special d}.
Thus it is of class $C^{0,1}$ in the whole $\Omega$.

To conclude the proof of proposition~\proref{MNT-Light}, we now
check that for every $u\in H\left(\Omega\right)$, such that $Lu=0$
in $\Omega$, there holds $\JacSJ\left(u,\mathcal{F}\right)=\left(\left(A\nabla u+bu\right)\cdot f_{1},\nabla u\cdot f_{2},\cdots,\nabla u\cdot f_{d},u\right)$
is of class $C^{0,\alpha}$ in $\Omega$. Note that for each $H_{i}\left(\overline{B_{r_{i}^{*}}\setminus B_{r_{i}}}\right)$,
there exists only one $j\in\left\{ 1,\cdots,N+1\right\} \setminus\left\{ i\right\} $
such that $\Omega_{j}\cap H_{i}\left(\overline{B_{r_{i}^{*}}\setminus B_{r_{i}}}\right)\neq\emptyset$
and $\Gamma_{ij}=H_{i}\left(\partial B_{2}\right)\subset H_{i}\left(\overline{B_{r_{i}^{*}}\setminus B_{r_{i}}}\right)$.
Thanks to the continuity of the flux $\left(ADu+bu\right)\cdot n=\left(ADu+bu\right)\cdot f_{1}$
on $\Gamma_{ij}$ , the Lipschitz continuity of $\mathcal{F}$ ,the
$C^{0,\alpha}$ continuity of $Du,u,A$ and $B$ in $\Omega_{i}$
or $\Omega_{j}$, we conclude that $\JacSJ\left(u,\mathcal{F}\right)$
is of class $C^{0,\alpha}$ in each $H_{i}\left(\overline{B_{r_{i}^{*}}\setminus B_{r_{i}}}\right)$
and $Q$. Moreover, we note that on each $\partial H_{i}\left(B_{r_{i}^{*}}\setminus B_{r_{i}}\right)$,
the coefficients $A$ and $b$ are uniformly $C^{0,\alpha}$, as they
are in the interior of $\Omega_{i}$ or $\Omega_{j}$. Therefore,
we have $\JacSJ\left(u,\mathcal{F}\right)$ is of class $C^{0,\alpha}$
on $\partial Q\setminus\partial\Omega=\cup_{i}\partial H_{i}\left(B_{r_{i}^{*}}\setminus B_{r_{i}}\right)$
(Note that for different $k$ and $s$, $\partial H_{k}\left(B_{r_{k}^{*}}\setminus B_{r_{k}}\right)\cap\partial H_{s}\left(B_{r_{s}^{*}}\setminus B_{r_{s}}\right)=\emptyset$).
In particular it is continuous. Thus $\JacSJ\left(u,\mathcal{F}\right)$
is of class $C^{0,\alpha}$ on $\Omega$.
\end{proof}

\subsection{\emph{\noun{Proof of Proposition~\proref{MNT-Light} for other dimensions.}}}
\begin{proof}
Let $\Pi_{i}$ be the geometric complement of $\Omega_{i}$, where
$i\in\left\{ 1,\cdots,N\right\} $. There exists a $C^{1,1}$ diffeomorphism
$H_{i}$ : $B_{2}\stackrel{H_{i}}{\rightarrow}\Pi_{i}$, which induces
a $C^{0,1}$ bijection on the vector fields: $DH_{i}:C^{0,1}\left(B_{2};\mathbf{\R}^{d}\right)\rightarrow C^{0,1}\left(\Pi_{i};\mathbf{\R}^{d}\right)$.
It is a map which maps $SO_{d}$ to $SO_{d}$ since the degree of
$H_{i}$ is either $1$ or $-1$ and moreover it maps the tangent
vectors (respectively the normal vector) on the sphere to the tangent
vectors (respectively the normal vector) on $\partial\Pi_{i}$, which
is the outer boundary of $\Omega_{i}$. Take $B_{r_{i}}\subset B_{2}\subset B_{r_{i}^{*}},1<r_{i}<2<r_{i}^{*}$,
such that 
\[
g_{H_{i}\left(B_{r_{i}}\right)}+1=g_{\Pi_{i}}=g_{H_{i}\left(B_{r_{i}^{*}}\right)}
\]

and such that the genus of the $\Pi_{i}\setminus H_{i}\left(B_{r_{i}}\right)$
equals to the genus of $H_{i}\left(B_{r_{i}^{*}}\right)\setminus\Pi_{i}$,
and equals one. In particular any $\Omega_{j}$, $j\neq i$, contained
in $\Pi_{i}$ are contained in $H_{i}\left(B_{r_{i}}\right)$ and
$H_{i}\left(B_{r_{i}^{*}}\right)$.

For any $M=\left(m_{ij}\right)_{\left(d+1\right)\times\left(d+1\right)}\in\R^{d+1}\times\R^{d+1},$
we write $\mathcal{P}\left(M\right)=\left(m_{i,j}\right)_{\left(d+1\right)\times d}.$
Thanks to \lemref{UnitBall}, we construct $\left\{ f_{1},\cdots,f_{d+1}\right\} \in C^{0,1}\left(\overline{H_{i}\left(B_{r_{i}^{*}}\right)\setminus H_{i}\left(B_{r_{i}}\right)};\mathbb{R}^{d}\right)^{d+1}$
with rank equals to $d$ as follows:
\begin{enumerate}
\item $\left\{ f_{1},\cdots,f_{d+1}\right\} =\left\{ DH_{i}\left(h_{1}\right),\cdots,DH_{i}\left(h_{d+1}\right)\right\} $
on $\partial\Pi_{i}$
\item On $H_{i}\left(\partial B_{r_{i}}\right)$ and $H_{i}\left(\partial B_{r_{i}^{*}}\right)$,
let $\left\{ f_{1},\cdots,f_{d+1}\right\} =\mathcal{P}\left(I_{d+1}\right)$
\item There exists a $C^{0,1}$ path $S:\partial B_{r_{i}}\times[r_{i},r_{i}^{*}]\rightarrow SO_{d+1}$
such that $S(x,r_{i})=DH_{i}^{-1}\left(I_{d+1}\right)$, $S(2)=DH_{i}^{-1}\left(\HH_{d}\left(x\right)\right)$
(where $\mathbf{H}_{d}$ is given in \eqref{def_Hd}) and $S\left(r_{i}^{*}\right)=DH_{i}^{-1}\left(I_{d+1}\right)$.
For any $r\in\left(r_{i},r_{i}^{*}\right)$ and $x\in\partial B_{r_{i}}$,
take $\left(f_{1},\ldots,f_{d}\right)(H_{i}\left(\frac{rx}{r_{i}}\right))=\mathcal{P}\left(DH_{i}\left(S\left(x,r\right)\right)\right)$.
\end{enumerate}
Since for any $x\in\partial B_{r_{i}}$ and $r\in\left[r_{i},r_{i}^{*}\right],$$S\left(x,r\right)\in SO_{d+1}$.
We have $\rank S\left(x,r\right)=d+1$. Therefore, $\rank\mathcal{P}S\left(x,r\right)=d$.
As before, we conclude that $\left\{ f_{1},\cdots,f_{d+1}\right\} $
is also of class $C^{0,1}$ in $\overline{H_{i}\left(B_{r_{i}^{*}}\right)\setminus H_{i}\left(B_{r_{i}}\right)}$.

Note that for each $i\in\left\{ 1,\cdots,N\right\} $, we have $\left(f_{1},\cdots,f_{d+1}\right)=\mathcal{P}\left(I_{d+1}\right)$
on $\partial\left(H_{i}\left(B_{r_{i}^{*}}\right)\setminus H_{i}\left(B_{r_{i}}\right)\right)$.

Set 
\begin{equation}
\left(f_{1},\cdots,f_{d+1}\right)=\mathcal{P}\left(I_{d+1}\right)\text{ in \ensuremath{Q\coloneqq\Omega\setminus\cup_{i=1}^{N}\overline{H_{i}\left(B_{r_{i}^{*}}\right)\setminus H_{i}\left(B_{r_{i}}\right)}} }\label{eq:constructionInOtherPart}
\end{equation}
 As we proved before, in each $\overline{H_{i}\left(B_{r_{i}^{*}}\right)\setminus H_{i}\left(B_{r_{i}}\right)}$,
$\left\{ f_{1},\cdots,f_{d}\right\} $ is of class $C^{0,1}$. It
is continuous on $\partial\left(H_{i}\left(B_{r_{i}^{*}}\right)\setminus H_{i}\left(B_{r_{i}}\right)\right)$
and Lipschitz continuous in $Q$ thanks to \eqref{constructionInOtherPart},
and therefore of $C^{0,1}$ globally on $\Omega$.

The rest of the proof is identical to that given when $d=2,4$ or
$8.$
\end{proof}

\newcommand{\etalchar}[1]{$^{#1}$}

\appendix

\specialsection{Additional Proofs}

\subsection{\label{subsec:ProofLemTrans}Proof of \lemref{translation_to_get_well-posedness}}
\begin{proof}[Proof of \lemref{translation_to_get_well-posedness}]
 Given $v\in H_{0}^{1}\left(\Omega\right),$there holds, using the
a priori bounds \assuref{regularitecoeff_}, Cauchy-Schwarz and completing
a square,
\begin{eqnarray*}
\left\langle Lv,v\right\rangle _{H^{-1}\left(\Omega\right)\times H^{1}\left(\Omega\right)} & = & \int_{\Omega}ADu\cdot Du+\left(b+c\right)Du\cdot u+qu^{2}\text{d}x\\
 & \geq & \lambda\left\Vert Du\right\Vert _{L^{2}\left(\Omega\right)}^{2}-2\lambda^{-1}\left\Vert Du\right\Vert _{L^{2}\left(\Omega\right)}\left\Vert u\right\Vert _{L^{2}\left(\Omega\right)}-\lambda^{-1}\left\Vert u\right\Vert _{L^{2}\left(\Omega\right)}^{2}\\
 & \geq & \lambda\left(\left\Vert Du\right\Vert _{L^{2}\left(\Omega\right)}-\lambda^{-2}\left\Vert u\right\Vert _{L^{2}\left(\Omega\right)}\right)^{2}-\left(\lambda^{-1}+\lambda^{-3}\right)\left\Vert u\right\Vert _{L^{2}\left(\Omega\right)}^{2}.
\end{eqnarray*}
Thus writing $M=\lambda^{-1}+\lambda^{-3}+1$, for any $i_{1},\cdots,i_{N+1}\in\left\{ 1,\cdots,N+1\right\} ^{N+1}$,
all Dirichlet boundary value problems associated with $L\left[i_{1},\cdots,i_{N+1}\right]+MI_{d}$
are well-posed in $\Omega$. If the Dirichlet boundary value problem
associated with $L_{i}:=L\left[i_{1},\cdots,i_{N+1}\right]$ is not
well-posed, there exists a non-zero solution of 
\[
\begin{cases}
L_{i}u=0 & \text{in \ensuremath{\Omega}}\\
u=0 & \text{on \ensuremath{\partial\Omega}}
\end{cases}
\]
Consider $\left(L_{i}+MI_{d}\right)^{-1}$ as a linear operator from
$L^{2}\left(\Omega\right)$ to $L^{2}\left(\Omega\right)\cap H_{0}^{1}\left(\Omega\right)$.
The ill-posedness of $L_{i}$ implies that $M^{-1}\in\sigma\left(\left(L_{i}+MI_{d}\right)^{-1}\right)$.
Thanks to the Rellich--Kondrachov embedding, $\left(L_{i}+MI_{d}\right)^{-1}:L^{2}\left(\Omega\right)\to L^{2}\left(\Omega\right)$
is a compact linear operator acting on $L^{2}\left(\Omega\right)$,
therefore $M^{-1}$ is an isolated eigenvalue. That is, there exists
$\aleph_{\left[i_{1},\cdots,i_{N+1}\right]}^{1}>0$ such that $B\left(M^{-1},\aleph_{\left[i_{1},\cdots,i_{N+1}\right]}^{1}\right)\setminus\{M^{-1}\}\subset\rho\left(\left(L_{i}+MI_{d}\right)^{-1}\right)$
.

When the Dirichlet boundary value problem is well-posed, $M^{-1}\in\rho\left(\left(L_{i}+MI_{d}\right)^{-1}\right)$.
The resolvent is open, thus there exists some $\aleph_{\left[i_{1},\cdots,i_{N+1}\right]}^{2}>0$
such that $B\left(M^{-1},\aleph_{\left[i_{1},\cdots,i_{N+1}\right]}^{2}\right)\subset\rho\left(\left(L_{i}+MI_{d}\right)^{-1}\right)$.

Now define 
\[
\aleph=\min_{i_{1},\cdots,i_{N+1}\in\left\{ 1,\cdots,N+1\right\} }\left(\aleph_{\left[i_{1},\cdots,i_{N+1}\right]}^{1},\aleph_{\left[i_{1},\cdots,i_{N+1}\right]}^{2}\right),\text{ and }\vartheta=\frac{\aleph M^{2}}{1+\aleph M}.
\]
We verify that for every $\translation\in\left(0,\vartheta\right)$
$M^{-1}\not\in\sigma\left(\left(L_{i}+\translation+M\right)^{-1}\right)$,
which in turn means that $L_{i}+\translation$ is well posed.
\end{proof}

\subsection{\label{subsec:Proof-of-Poincare}Proof of \lemref{Poincare-Annulus}}
\begin{fact*}
There exists some $\eta_{0}>0$, such that for any $0<\eta<\eta_{0},$
and any $t\in\left(\frac{1}{2},1\right)$ there holds
\begin{equation}
\forall u\in H_{0}^{1}\left(U^{t}\right),\left\langle L_{k}u,u\right\rangle _{H^{-1}\left(U^{t}\right),H_{0}^{1}\left(U^{t}\right)}\geq\frac{1}{3}\lambda\left\Vert u\right\Vert _{H_{0}^{1}\left(U^{t}\right)}^{2}.\label{eq:coerciveL-1}
\end{equation}
\end{fact*}
\begin{proof}
Indeed, we have, for any $t>0$,
\begin{eqnarray}
 &  & \left\langle L_{k}u,u\right\rangle \nonumber \\
 & = & \int_{U^{t}}A\nabla u\cdot\nabla u+u\left(b+c\right)\cdot\nabla u+qu^{2}\text{d}x\nonumber \\
 & \geq & \lambda\left\Vert \nabla u\right\Vert _{L^{2}\left(U^{t}\right)}^{2}-2\lambda^{-1}\int_{U^{t}}\left|\nabla u\right|\left|u\right|\text{d}x-\lambda^{-1}\left\Vert u\right\Vert _{L^{2}\left(U^{t}\right)}^{2}\nonumber \\
 & \geq & \frac{\lambda}{2}\left\Vert \nabla u\right\Vert _{L^{2}\left(U^{t}\right)}^{2}-\frac{\lambda^{2}+2}{\lambda^{3}}\left\Vert u\right\Vert _{L^{2}\left(U^{t}\right)}^{2}.\label{eq:estimate-for-coercivity-in-a-small-domain}
\end{eqnarray}
To address the lower order term we rely on \lemref{Poincare-Annulus}.
Since $U^{t}=\psi_{k}^{-1}\left(B_{\frac{1}{t}}\setminus B_{\frac{1}{t}\left(1-\eta\right)}\right)$,
by changing variables, \lemref{Poincare-Annulus} shows that for any
$u\in H_{0}^{1}\left(U^{t}\right)$ there holds
\begin{equation}
\left\Vert u\right\Vert _{L^{2}\left(U^{t}\right)}^{2}\leq C\frac{\eta^{2}}{t^{2}}\left\Vert \nabla u\right\Vert _{L^{2}(U^{t})}^{2}\leq4C\eta^{2}\left\Vert \nabla u\right\Vert _{L^{2}(U^{t})}^{2}.\label{eq:estimate-poincare}
\end{equation}

Combining \eqref{estimate-for-coercivity-in-a-small-domain} and \eqref{estimate-poincare},
we have 
\[
\left\langle Lu,u\right\rangle _{H^{-1}\left(U^{t}\right)\times H_{0}^{1}\left(U^{t}\right)}\geq\left(\frac{\lambda}{2}-4C\frac{\lambda^{2}+2}{\lambda^{3}}\eta^{2}\right)\left\Vert \nabla u\right\Vert _{L^{2}\left(U^{t}\right)}^{2}.
\]
Choosing $\eta>0$ small enough there holds for all $t\in\left(\frac{1}{2},1\right)$,
\[
\left\langle Lu,u\right\rangle _{H^{-1}\left(U^{t}\right)\times H^{1}\left(U^{t}\right)}\geq\frac{1}{3}\lambda\left\Vert \nabla u\right\Vert _{L^{2}\left(U^{t}\right)}^{2}.
\]
\end{proof}
\begin{lem}
\label{lem:Poincare-Annulus}Write $B_{r}$ for the ball centred at
the origin of radius $r$. Given $0<r_{2}<r_{1}$ , for any $s$ and
$t$ such that $r_{1}<t<s<r_{2}$, there holds 
\[
\forall u\in H_{0}^{1}\left(B_{s}\setminus B_{t}\right)\ \left\Vert u\right\Vert _{L^{2}\left(B_{s}\setminus B_{t}\right)}^{2}\leq c\left(s-t\right)^{2}\left\Vert \nabla u\right\Vert _{L^{2}(B_{s}\setminus B_{t})}^{2}.
\]
for some constant $c$, which depends on $r_{1}$ and $r_{2}$ only.
\end{lem}

\begin{proof}
Consider the Dirichlet eigenvalue problem in $B_{s}\setminus B_{t}$
\[
\begin{cases}
\triangle u=\rho_{st}u & \text{in \ensuremath{B_{s}\setminus B_{t}} }\\
u=0 & \text{on \ensuremath{\partial B_{s}}}\\
u=0 & \text{on \ensuremath{\partial B_{t}}}
\end{cases}
\]

We note that the first eigensolution is radial, $u=f_{st}\left(\left|r\right|\right),$
and $f_{st}\in C^{\infty}\left(\left(t,s\right)\right)$ satisfies
\[
\frac{1}{r^{d-1}}\partial_{r}\left(r^{d-1}\partial_{r}f_{st}\right)=\rho_{st}^{1}f\text{ in }\left(t,s\right)\quad f_{st}\left(s\right)=f_{st}\left(t\right)=0.
\]
By the change of variable $r\to\frac{r_{2}-r_{1}}{s-t}\left(r-t\right)+r_{1},$
we find that $f_{st}\left(r\right)=f_{r_{2}r_{1}}\left(\frac{r_{1}-r_{2}}{s-t}\left(r-t\right)+r_{2}\right)$,
and $\rho_{st}^{1}=\left(\frac{r_{2}-r_{1}}{s-t}\right)^{2}\rho_{r_{2}r_{1}}^{1}.$
\[
\rho_{st}^{1}=\inf_{\underset{u\neq0}{u\in H_{0}^{1}\left(B_{s}\setminus B_{t}\right)}}\text{\ensuremath{\frac{\int_{B_{s}\setminus B_{t}}\nabla u\cdot\nabla u\text{d}x}{\int_{B_{s}\setminus B_{t}}u^{2}\text{d}x}}=\ensuremath{\inf_{\underset{u\neq0}{u\in H_{0}^{1}\left(B_{s}\setminus B_{t}\right)}}\text{\ensuremath{\frac{\left\Vert \nabla u\right\Vert _{L^{2}\left(B_{s}\setminus B_{t}\right)}^{2}}{\left\Vert u\right\Vert _{L^{2}\left(B_{s}\setminus B_{t}\right)}^{2}}}}},}
\]

We conclude that $\left\Vert u\right\Vert _{L^{2}\left(B_{s}\setminus B_{t}\right)}^{2}\leq\left(\rho_{r_{1}r_{2}}^{1}\right)^{-1}\left(\frac{s-t}{r_{1}-r_{2}}\right)^{2}\left\Vert \nabla u\right\Vert _{L^{2}(B_{s}\setminus B_{t})}^{2}$
for every $u\in H_{0}^{1}\left(B_{s}\setminus B_{t}\right)$.
\end{proof}

\subsection{\label{subsec:PfPaFx}Proof of \lemref{Whitney-Reduction-Lemma-Light}}
\begin{proof}
We have 
\[
\left(\begin{array}{c}
\JacSJ\left(u_{1},\mathcal{F}\right)\\
\vdots\\
\JacSJ\left(u_{P},\mathcal{F}\right)
\end{array}\right)=\JacJ\left(u_{1},\cdots,u_{P}\right)^{T}T\left(x,E_{d+1,d}f_{1},\cdots,E_{d+1,d}f_{\dstar},\text{e}_{d+1}\right).
\]
Thanks to proposition~\proref{MNT-Light} there holds $\text{rank}\left(f_{1},\cdots,f_{\dstar}\right)=d$.
Furthermore 
\[
\left(E_{d+1,d}f_{1},\cdots,E_{d+1,d}f_{\dstar}\right)\cap\mathbb{R}e_{d+1}=\left\{ 0\right\} ,
\]
thus proposition~\propref{current-matrix} shows that $\text{rank}\left(T\left(x,E_{d+1,d}f_{1},\cdots,E_{d+1,d}f_{\dstar},\text{e}_{d+1}\right)\right)=d+1$. 

Since $\left\{ u_{1},\cdots,u_{P}\right\} \in\mathcal{A}\left(P\right)$,
we have $\rank\JacJ\left(u_{1},\cdots,u_{P}\right)^{T}=d+1$ at every
$x$, thus $\rank F_{x}=d+1$.

Note that $\forall a\in\mathbb{R}^{P-1}$ $\rank P_{a}=P-1$ thus
for every $x$, we have:
\[
\rank P_{a}\circ F_{x}\leq\min\left(\rank P_{a},\rank F_{x}\right)\leq d+1
\]

and 
\[
\rank P_{a}\circ F_{x}\geq\rank P_{a}+\rank F_{x}-P=d
\]

If $a\in\R^{P-1}\setminus G$, then there exists $x\in\Omega$, such
that
\begin{eqnarray}
 &  & \rank P_{a}\circ F_{x}=d\nonumber \\
 & \Longleftrightarrow & \dim\ker\left(P_{a}\circ F_{x}\right)=\dstar+1-d\label{eq:PaFx}\\
 & \Longleftrightarrow & \dim F_{x}^{-1}\left(\mathop{{\rm span}}\left\{ \left(a_{1},\cdots,a_{P-1},1\right)\right\} \right)=\dstar+1-d\nonumber 
\end{eqnarray}

We have the implication $a\in\R^{P-1}\setminus G\Longrightarrow$$\left(a_{1},\cdots,a_{P-1},1\right)\in\cup_{x}\Im\left(F_{x}\right)$.
Conversely, if $\left(a_{1},\cdots,a_{P-1},1\right)\in\cup_{x}\Im\left(F_{x}\right)$
then there exists $x\in\Omega\setminus\cup_{i,j}\Gamma_{ij}$ and
$v_{a}\in\R^{\dstar+1}$ such that $F_{x}v_{a}=\left(a_{1},\cdots,a_{P-1},1\right)$.
Thus, 
\[
\mathbb{R}v_{a}\oplus\ker\left(F_{x}\right)\subset F_{x}^{-1}\left(\text{span}\left\{ \left(a_{1},\cdots,a_{P-1},1\right)\right\} \right).
\]
Even though the choice of $v_{a}$ is arbitrary, any other choice
would be in $\mathbb{R}v_{a}\oplus\ker\left(F_{x}\right)$, thus $\mathbb{R}v_{a}\oplus\ker\left(F_{x}\right)=F_{x}^{-1}\left(\text{span}\left\{ \left(a_{1},\cdots,a_{P-1},1\right)\right\} \right)$.
Note that since $\rank\left(F_{x}\right)=d+1$, $\dim\left(\ker\left(F_{x}\right)\right)=\dstar-d$,
therefore $\dim\left(\mathbb{R}v_{a}\oplus\ker\left(F_{x}\right)\right)=\dstar+1-d$,
which from \eqref{PaFx} implies that $\rank P_{a}\circ F_{x}=d$.

In conclusion, we have $P_{a}\circ F_{x}$ has rank $d+1$$\Longleftrightarrow\left(a_{1},\ldots a_{P-1},1\right)\not\in\cup_{x}\Im\left(F_{x}\right)$.
Set $B=\cup_{x}\Im\left(F_{x}\right)\cap\left\{ b\in\R^{P}|b_{P}=1\right\} $.
The identity $\R^{P-1}\setminus G=P_{P-1,P}\left(B\right)$ therefore
holds. We now follow the argument in \cite[Lemma 4.1]{2020-Alberti-Capdeboscq-IMRN}
and \cite{Greene-Wu-1975a} and deduce that $\mathcal{H}^{k-1}\left(B\right)=0$.
The conclusion is attained as the $P-1$-Hausdorff measure equals
to the $P-1$-Lebesgue measure.
\end{proof}

\section{\label{sec:Kervaire}Proof of \lemref{UnitBall}}

When $d\not\in\left\{ 2,4,8\right\} $ it is impossible to find a
group of continuous vector fields family $\left\{ h_{1},\ldots,h_{d}\right\} $
such that for every $x\in\partial B_{1}$, there holds
\begin{enumerate}
\item $h_{1}\left(x\right)=x\text{ ,}$
\item $\left\langle h_{i},h_{j}\right\rangle \left(x\right)=\delta_{ij}$
$\text{ for \ensuremath{i,j=1,\ldots d.}}$
\end{enumerate}
In odd dimensions, this is a consequence, of the so-called \emph{Hairy
ball }theorem. In general, the following result is proved in \cite{kervaire1958non}
and \cite{bott1958parallelizability}.
\begin{thm*}
\label{thm:Michel Kervaire} There exists trivial bundle of $S^{d-1}$
if and only if $d=2,4$ or $8$. Moreover, when $d\in\left\{ 2,4,8\right\} $
there exists $\left\{ h_{2},\cdots,h_{d}\right\} \in\left(C^{1}\left(S^{d-1};\R^{d}\right)\right)^{d-1}$
such that $\left(h_{1},\cdots,h_{d}\right)\in SO_{d}\left(S^{d-1}\right)$
where $SO_{d}$ refers to the real unitary matrices with positive
determinant.
\end{thm*}
Explicit examples are :
\begin{enumerate}
\item When $d=2$, $\forall\left(x_{1},x_{2}\right)\in\partial B_{1}$,
set $h_{1}=\left(x_{1},x_{2}\right)$ and $h_{2}=\left(-x_{2},x_{1}\right)$.
\item When $d=4$, $\forall\left(x_{1},x_{2},x_{3},x_{4}\right)\in\partial B_{1}$,
set
\begin{eqnarray*}
h_{1} & = & \left(x_{1},x_{2},x_{3},x_{4}\right),\\
h_{2} & = & \left(-x_{2},x_{1},-x_{4},x_{3}\right),\\
h_{3} & = & \left(x_{3},-x_{4},-x_{1},x_{2}\right),\\
h_{4} & = & \left(x_{4},x_{3},-x_{2},-x_{1}\right).
\end{eqnarray*}
\item When $d=8$, $\forall\left(x_{1},x_{2},x_{3},x_{4},x_{5},x_{6},x_{7},x_{8}\right)\in\partial B_{1}$
set
\begin{eqnarray*}
h_{1} & = & \left(x_{1},x_{2},x_{3},x_{4},x_{5},x_{6},x_{7},x_{8}\right),\\
h_{2} & = & \left(-x_{2},x_{1},-x_{4},x_{3},-x_{6},x_{5},x_{8},-x_{7}\right),\\
h_{3} & = & \left(-x_{3},x_{4},x_{1},-x_{2},-x_{7},-x_{8},x_{5},x_{6}\right),\\
h_{4} & = & \left(-x_{4},-x_{3},x_{2},x_{1},-x_{8},x_{7},-x_{6},x_{5}\right),\\
h_{5} & = & \left(-x_{5},x_{6},x_{7},x_{8},x_{1},-x_{2},-x_{3},-x_{4}\right),\\
h_{6} & = & \left(-x_{6},-x_{5},x_{8},-x_{7},x_{2},x_{1},x_{4},-x_{3}\right),\\
h_{7} & = & \left(-x_{7},-x_{8},-x_{5},x_{6},x_{3},-x_{4},x_{1},x_{2}\right),\\
h_{8} & = & \left(-x_{8},x_{7},-x_{6},-x_{5},x_{4},x_{3},-x_{2},x_{1}\right).
\end{eqnarray*}
\end{enumerate}
The second part of \lemref{UnitBall} follows from the following proposition.
\begin{prop*}
There exists $h_{2},\cdots,h_{d+1}$ in $\left(C^{1}\left(S^{d-1},\mathbf{\R}^{d}\right)\right)^{d}$
such that $\left\langle h_{i},x\right\rangle =0$, for $i=2,\cdots,d+1$
and $\rank\left(x,h_{2},\cdots,h_{d+1}\right)=d$ on $S^{d-1}$.
\end{prop*}
\begin{proof}
For every $x\in S^{d-1}\subset\R^{d}$, we denote $x=\left(x_{1},x_{2},\cdots,x_{d}\right)$.
Set 
\[
h_{i}=\left(x_{1}x_{d+2-i}-\delta_{1,d+2-i},\cdots,x_{d}x_{d+2-i}-\delta_{d,d+2-i}\right),
\]
 where $\delta_{i,j}$ is the Kronecker symbol. We have $\left\langle h_{i},x\right\rangle =\left(\sum_{j=1}^{d}x_{j}^{2}x_{d+2-i}\right)-x_{d+2-i}=0$,
for $i\geq2$, thus each $h_{i}$ is tangent to $S^{d-1}$. Take
\begin{equation}
\HH_{d}=\left(\begin{array}{cc}
h_{1} & 1\\
h_{2} & x_{d}\\
\vdots & \vdots\\
h_{d+1} & x_{1}
\end{array}\right)_{\left(d+1\right)\times\left(d+1\right)}\text{, that is, }\HH_{d}=\left(\begin{array}{ccccc}
x_{1} & x_{2} & \ldots & x_{d} & 1\\
x_{1}x_{d} & x_{2}x_{d} & \ldots & x_{d}^{2}-1 & x_{d}\\
\vdots & \vdots & \vdots & \vdots & \vdots\\
x_{1}x_{2} & x_{2}^{2}-1 & \ldots & x_{d}x_{2} & x_{2}\\
x_{1}^{2}-1 & x_{1}x_{2} & \ldots & x_{d}x_{1} & x_{1}
\end{array}\right).\label{eq:def_Hd}
\end{equation}
There holds $\rank\HH_{d}=d+1$, for $d\geq2.$ The proof is by induction.
When $d=2$, we compute $\det\HH_{2}=-1.$ When $d\geq3$, we have
\[
\det\HH_{d}=\left|\begin{array}{ccccc}
x_{1} & x_{2} & \ldots & x_{d} & 1\\
0 & 0 & \ldots & -1 & 0\\
\vdots & \vdots & \vdots & \vdots & \vdots\\
x_{1}x_{2} & x_{2}^{2}-1 & \ldots & x_{d}x_{2} & x_{2}\\
x_{1}^{2}-1 & x_{1}x_{2} & \ldots & x_{d}x_{1} & x_{1}
\end{array}\right|=\left(-1\right)^{d+1}\det\HH_{d-1}=\ldots=\left(-1\right)^{\frac{d\left(d+3\right)}{2}}.
\]
Thus, we have $\rank\HH_{d}=d+1$, which implies $\rank\left(h_{1},\cdots,h_{d+1}\right)=d$.
We modify $h_{d+1}\rightarrow h_{d+1}\left(-1\right)^{\frac{d\left(d+3\right)}{2}}$
and modify the last line of $\HH_{d}$ to be $\left(h_{d+1}\left(-1\right)^{\frac{d\left(d+3\right)}{2}},\left(-1\right)^{\frac{d\left(d+3\right)}{2}}x_{1}\right)$
to obtain $\HH_{d}\in SO_{d+1}$.
\end{proof}

\section{\label{sec:EndProofMR}Proof of \thmref{mainresult}}
\begin{proof}
We reproduce the proof given in \cite{2020-Alberti-Capdeboscq-IMRN}
with the necessary adaptations for the reader's convenience. Thanks
to \thmref{piecewise_regular-construction}, and in turn \eqref{RoughNumber},
there exists a large $P_{0}$ such that $\mathcal{A}\left(P_{0}\right)\neq\emptyset$.
Write $P^{\star}=\left[\frac{d+\dstar+1}{\alpha}\right]$. Take $h\in H\left(\Omega\right)^{P^{\star}}$,
namely $h=\left(h_{1},\cdots,h_{P^{\star}}\right)$. Take $u_{1},\cdots,u_{P_{0}}\in\mathcal{A}\left(P_{0}\right).$
Then $\left(h_{1},\cdots,h_{P^{\star}},u_{1},\cdots,u_{P_{0}}\right)\in\mathcal{A}\left(P_{0}+P^{\star}\right)$,
and for $x\in\Omega\setminus\cup_{i\neq j}\Gamma_{ij},$
\[
\rank\JacJ\left(h_{1},\cdots,h_{P^{\star}},u_{1},\cdots,u_{P_{0}}\right)\left(x\right)=d+1.
\]
Thanks to \lemref{Main-reduction-lemma}, for a.e $a^{P_{0}+P^{\star}-1}\in\mathbb{R}^{P_{0}+P^{\star}-1}$
, there holds
\[
\rank\JacJ\left(h_{1}-a_{1}^{P_{0}+P^{\star}-1}u_{P_{0}},\cdots,h_{P^{\star}}-a_{P^{\star}}^{P_{0}+P^{\star}-1}u_{P_{0}},\cdots,u_{P_{0}-1}-a_{P_{0}+P^{\star}-1}^{P_{0}+P^{\star}-1}u_{P_{0}}\right)\left(x\right)=d+1.
\]
Repeating this reduction $P_{0}$ times, for a.e $a^{T}=\left(a_{1}^{T},\cdots,a_{T}^{T}\right)\in\mathbb{R}^{T}$,
where $T=\left(P^{\star},\cdots,P_{0}+P^{\star}-1\right)$, there
holds
\[
\rank\JacJ\left(h_{1}-\sum_{T=P^{\star}}^{P_{0}+P^{\star}-1}a_{1}^{T}u_{T-P^{\star}+1},\cdots,h_{P^{\star}}-\sum_{T=P^{\star}}^{P_{0}+P^{\star}-1}a_{P^{\star}}^{T}u_{T-P^{\star}+1}\right)\left(x\right)=d+1,
\]
which means $h_{a^{T}}=\left(h_{1}-\sum_{T=P^{\star}}^{P_{0}+P^{\star}-1}a_{1}^{T}u_{T-P^{\star}+1},\cdots,h_{P^{\star}}-\sum_{T=P^{\star}}^{P_{0}+P^{\star}-1}a_{P^{\star}}^{T}u_{T-P^{\star}+1}\right)\in\mathcal{A}\left(P^{\star}\right)$.
For any $\epsilon>0$, taking $a^{T}$ small enough, since $u_{1},\cdots,u_{P_{0}}$
are bounded in $H\left(\Omega\right)$, we conclude that 
\[
\left\Vert h-h_{a^{T}}\right\Vert _{H\left(\Omega\right)^{P^{\star}}}\leq\epsilon.
\]
We then prove that $\mathcal{A}\left(P^{\star}\right)$ is an open
set. For any $x\in\overline{\Omega},\mathbf{u=}\left(u_{1},\cdots,u_{P^{\star}}\right)\in H\left(\Omega\right)^{P^{\star}}$,
we define $\mathbf{Det}:\text{\ensuremath{\overline{\Omega}}}\times H\left(\Omega\right)^{P^{\star}}\rightarrow\mathbb{R}$
given by
\[
\mathbf{Det}\left(x,\mathbf{u}\right)\coloneqq\sum_{i_{1},\cdots,i_{d+1}=1}^{P^{\star}}\text{det}\left|\left(\JacSJ\left(u_{i_{1}},\cdots,u_{i_{d+1}}\right)\right)\left(x\right)\right|
\]

Thanks to the continuity and the multi-linearity of $\JacSJ$, $\mathbf{Det}\left(x,\mathbf{u}\right)$
is continuous for every $x\in\overline{\Omega},\mathbf{u=}\left(u_{1},\cdots,u_{P^{\star}}\right)\in H\left(\Omega\right)^{P^{\star}}$.
Take $\mathbf{u}\in\mathcal{A}\left(P^{\star}\right)$, for every
$x\in\overline{\Omega}$, there holds
\[
\mathbf{Det}\left(x,\mathbf{u}\right)>0
\]

Therefore, there exists some constant $C>0$ such that
\[
\text{inf}_{x\in\overline{\Omega}}\mathbf{Det}\left(x,\mathbf{u}\right)\geq C>0
\]

Take $\epsilon>0$ small enough and $\mathbf{v}=\left(v_{1},\cdots,v_{P^{\star}}\right)\in H\left(\Omega\right)^{P^{\star}}$
such that

\[
\left\Vert \mathbf{u}-\mathbf{v}\right\Vert _{H\left(\Omega\right)^{P^{\star}}}=\sum_{i=1}^{P^{\star}}\left\Vert u_{i}-v_{i}\right\Vert _{H\left(\Omega\right)}\leq\epsilon\text{ and }\mathbf{Det}\left(x,\mathbf{v}\right)\geq\frac{C}{2}>0,
\]
which implies $\rank\left(\JacSJ\left(v_{1},\cdots,v_{P^{\star}}\right)\right)=d+1$.
Thanks to the relation between $\JacJ$ and $\JacSJ$, we conclude
that $\rank\left(\JacJ\left(v_{1},\cdots,v_{P^{\star}}\right)\right)=d+1$
which implies $\mathbf{v}\in\mathcal{A}\left(P^{\star}\right)$. 
\end{proof}

\end{document}